\documentclass[a4paper,11pt,reqno]{amsart}
\usepackage{amsmath}
\usepackage{amsfonts}
\usepackage{accents}
\usepackage{bbm}
\usepackage{amssymb}
\usepackage{graphicx}
\usepackage{geometry}
\usepackage{wrapfig}
\usepackage{ textcomp }
\usepackage[small,nohug,heads=vee]{diagrams} 
\theoremstyle{plain}
\newtheorem{theorem}{Theorem}
\numberwithin{theorem}{section}

\newtheorem{claim}[theorem]{Claim}

\newtheorem{question}[theorem]{Question}

\newtheorem{definition}[theorem]{Definition}
\newtheorem{example}[theorem]{Example}

\newtheorem{lemma}[theorem]{Lemma}

\newtheorem{proposition}[theorem]{Proposition}
\newtheorem{remark}[theorem]{Remark}

\newtheorem*{theoremX}{Theorem}
\newtheorem*{exampleX}{Example}

\newcommand\myover[2]{\genfrac{}{}{0pt}{}{#1}{#2}}

\renewcommand{\H}{\mathrm{H}}
\newcommand{\DG}{{DG}}
\renewcommand{\mod}{\mathrm{mod}}
\newcommand{\comod}{\mathrm{comod}}
\newcommand{\Aut}{\mathrm{Aut}}
\newcommand{\Rep}{\mathrm{Rep}}
\newcommand{\Int}{\mathrm{Int}}
\newcommand{\Bigal}{\mathrm{Bigal}}
\newcommand{\BrPic}{\mathrm{BrPic}}
\newcommand{\Gal}{\mathrm{Gal}}
\newcommand{\End}{\mathrm{End}}

\newcommand{\Cent}{\mathrm{Cent}}
\newcommand{\Vect}{\mathrm{Vect}}
\newcommand{\Out}{\mathrm{Out}}
\newcommand{\Inn}{\mathrm{Inn}}
\newcommand{\Reg}{\mathrm{Reg}}
\newcommand{\Hom}{\mathrm{Hom}}
\newcommand{\id}{\mathrm{id}}
\newcommand{\im}{\mathrm{im}}

\newcommand{\DD}{\mathbb{D}}
\renewcommand{\SS}{\mathbb{S}}

\newcommand{\F}{\mathbb{F}}

\newcommand{\Z}{\mathrm{Z}}
\newcommand{\Ind}{\mathrm{Ind}}
\newcommand{\B}{\mathrm{B}}
\renewcommand{\P}{\mathrm{P}}
\newcommand{\GL}{\mathrm{GL}}
\newcommand{\Sp}{\mathrm{Sp}}
\newcommand{\SO}{\mathrm{SO}}
\renewcommand{\O}{\mathrm{O}}
\renewcommand{\d}{\mathrm{d}}
\newcommand{\ZZ}{\mathbb{Z}}
\newcommand{\CC}{\mathbb{C}}
\newcommand{\LL}{\mathcal{L}}
\newcommand{\nat}{\mathbb{N}}

\newcommand{\cat}{\mathcal{C}}

\newcommand{\ocat}{\mathcal{O}}

\newcommand{\mcat}{\mathcal{M}}
\newcommand{\bcat}{\mathcal{B}}

\newcommand{\md}{\text{-}}
\newcommand{\mref}{\mathrm{ref}}
\newcommand{\triv}{\mathrm{triv}}
\newcommand{\sgn}{\mathrm{sgn}}

\newcommand{\cB}{\mathcal{B}}
\newcommand{\cE}{\mathcal{E}}
\newcommand{\cV}{\mathcal{V}}
\newcommand{\cR}{\mathcal{R}}

\newcommand{\cBL}{\mathcal{B}_L}
\newcommand{\cEL}{\mathcal{E}_L}
\newcommand{\cVL}{\mathcal{V}_L}
\newcommand{\cRL}{\mathcal{R}_L}

\newcommand{\cBTildeL}{\widetilde{\mathcal{B}}_L}
\newcommand{\cETildeL}{\widetilde{\mathcal{E}}_L}
\newcommand{\cVTildeL}{\widetilde{\mathcal{V}}_L}
\newcommand{\cRTildeL}{\widetilde{\mathcal{R}}_L}

\newcommand{\cBUnderL}{\underaccent{\bar}{\mathcal{B}}_L}
\newcommand{\cEUnderL}{\underaccent{\bar}{\mathcal{E}}_L}
\newcommand{\cVUnderL}{\underaccent{\bar}{\mathcal{V}}_L}
\newcommand{\cRUnderL}{\underaccent{\bar}{\mathcal{R}}_L}

\newcommand{\AutTilde}{\widetilde{\mathrm{Aut}}}
\newcommand{\AutUnder}{\underline{\mathrm{Aut}}}

\newcommand{\ito}{\stackrel{\simeq}{\to}}
\newcommand{\hamburger}[4] 
{
  \thispagestyle{empty}
  \vspace*{-2cm}
  \begin{flushright}
    ZMP-HH #2 \\
    Hamburger Beitr\"age zur Mathematik Nr. #3 \\
    #4 \\
  \end{flushright}
  \vspace{0.5cm}
  \begin{center}
    \Large \bf
    #1
  \end{center}
  \vspace{0.5cm}
  \begin{center}        
    Simon Lentner, Jan Priel \\
    Fachbereich Mathematik, Universit\"at Hamburg \\
    Bereich Algebra und Zahlentheorie \\
    Bundesstra\ss e 55, D-20146 Hamburg \\
  \end{center}
  \vspace{0.5cm}
}

\begin{document}

\hamburger{A decomposition of the Brauer-Picard group of the representation category of a finite group}{15 / 9}{539}{June 2015}
\thispagestyle{empty}
\enlargethispage{1cm}

\begin{abstract}
We present an approach of calculating the group of braided autoequivalences of
the category of representations of the Drinfeld double of a finite dimensional Hopf algebra $H$ and
thus the Brauer-Picard group of $H\md\mod$. We consider two natural subgroups
and a subset as candidates for generators. In this article $H$ is the group algebra of a finite group $G$. As our main result we prove that any element of the Brauer-Picard group, fulfilling an additional cohomological condition, decomposes into an ordered product of our candidates. \\
For elementary abelian groups $G$ our decomposition reduces to the Bruhat
decomposition of the Brauer-Picard group, which is in this case a Lie group over
a finite field.  
Our results are motivated by and have applications to symmetries and defects in
$3d$-TQFT and group extensions of fusion categories.
\end{abstract}

\makeatletter
\@setabstract
\makeatother

\tableofcontents
\newpage

\section{Introduction}

For a finite tensor category $\cat$ the \emph{Brauer-Picard group}
$\BrPic(\cat)$ is defined as the group of equivalence classes of invertible
$\cat$-bimodule categories. This group is an important invariant of the tensor
category $\cat$ and appears at several essential places in representation
theory, for example in the classification problem of $G$-extensions of fusion
categories, see \cite{ENO10}. In mathematical physics (bi-)module categories
appear as boundary conditions and defects in $3d$-TQFT, in particular the
Brauer-Picard group is a symmetry group of such theories, see \cite{FSV13},
\cite{FPSV15}. \\
An important structural insight is the following result proven in Thm. 1.1 \cite{ENO10} for $\cat$ a fusion category and in Thm. 4.1 \cite{DN12} for $\cat$ a finite tensor category (hence not necessarily semisimple):
There is a group isomorphism from the Brauer-Picard group to the group of equivalence classes of \emph{braided autoequivalences}
of the \emph{Drinfeld center} $Z(\cat)$:
\begin{align}\label{ENO} \BrPic(\cat) \cong \Aut_{br}(Z(\cat)) \end{align} 

In the case $\cat=\Rep(G)$ of finite dimensional complex representations of a finite group $G$
(respectively $\cat=\Vect_G$ which has the same Drinfeld center) computing the Brauer-Picard group is already an interesting and non-trivial task. This group appears as the symmetry group of (extended) Dijkgraaf-Witten theories with structure group $G$. See \cite{DW90} for the original work on Chern-Simons with finite gauge group $G$ and see \cite{FQ93}, \cite{Mo13} for the extended case. In \cite{O03} the authors have obtained a parametrization of $\Vect_G$-bimodule
categories in terms of certain subgroups $L \subset G\times G^{op}$ and
$2$-cocycles $\mu$ on $L$ and \cite{Dav10} has determined a condition when such pairs correspond to invertible bimodule categories. However, the necessary calculations to determine
$\BrPic(\cat)$ seem to be notoriously hard and the above approach gives little information about the
group structure. In \cite{NR14} the authors use the isomorphism to $\Aut_{br}(Z(\cat))$ in order
to compute the Brauer-Picard group for several groups $G$ using the following strategy: They enumerate all
subcategories $\LL\subset Z(\cat)$ that are braided equivalent to $\cat=\Rep(G)$, then they prove that
$\Aut_{br}(Z(\cat))$ acts transitively on this set. Finally, they determine the stabilizer of
the standard subcategory $\cat \subset Z(\cat)$ with trivial braiding. 

For $G$ abelian, the second author's joint paper \cite{FPSV15} determines a set of generators of the
Brauer-Picard group and provides a field theoretic interpretation of the isomorphism $\BrPic(\cat) \cong \Aut_{br}(Z(\cat))$ in terms of 3d-Dijkgraaf-Witten theory with defects. Results for Brauer-Picard groups of other categories $\cat$ include representations of the Taft algebra in \cite{FMM14} and of supergroups in \cite{Mom12},\cite{BN14}. \\ 
An alternative characterization of elements in $\Aut_{br}(Z(H\md\mod))$ in terms of quantum commutative Bigalois objects was given in \cite{ZZ13}. \\

\enlargethispage{0.2cm}

In this article we propose an approach to calculate $\BrPic(\cat)$
for $\cat=H\md\mod$, the category of finite-dimensional representations of a
finite-dimensional Hopf algebra $H$. 
Let $\cat$ be any tensor category. Then there exists a well-known group homomorphism: 
$$ \Ind_\cat:\;\Aut_{mon}(\cat)\to \BrPic(\cat) \cong \Aut_{br}(Z(\cat)) $$
given by assigning to a monoidal automorphism $\Psi \in \Aut_{mon}(\cat)$ the invertible $\cat$-bimodule category $_{\Psi}\cat_\cat$, where the left $\cat$-module structure is given by precomposing with $\Psi$; then we use the isomorphism (\ref{ENO}) mentioned above. The image of this map gives us a natural subgroup of the Brauer-Picard group. If we can choose another category $\cat'$ and a braided equivalence $F: Z(\cat') \stackrel{\sim}{\to} Z(\cat)$, then we get a different induction and a new subgroup of $\Aut_{br}(Z(\cat))$:
$$\Ind_{\cat',F}:\;\Aut_{mon}(\cat')\to \BrPic(\cat') \cong \Aut_{br}(Z(\cat')) \stackrel{F}{\cong} \Aut_{br}(Z(\cat))$$
Consider a finite dimensional Hopf algebra $H$ and let $\cat=H\md\mod$ be the category of finite dimensional $H$-modules. Then $Z(H\md\mod)=DH\md\mod=H^*\bowtie H\md\mod$ and we have a canonical choice $\cat'=H^*\md\mod$ and a canonical isomorphism of Hopf algebras $DH \ito D(H^*)$ (see Thm. 3 in \cite{Rad93}), that gives us a canonical braided equivalence $D(H^*)\md\mod \cong DH\md\mod$. Hence, we have two canonical subgroups of $\Aut_{br}(DH\md\mod)$, namely $\im(\Ind_{H \md \mod)})$ and $\im(\Ind_{H^*\md\mod})$. \\

Let us introduce an additional set $\cR\subset \Aut_{br}(DH\md\mod)$. For each decomposition of $H$ into a Radford biproduct $H=A\ltimes K$ (see Sect. 10.6 \cite{Mon93}), where $A$ a Hopf algebra and $K$ a Hopf algebra in the category $Z(A\md\mod)$, for a choice of a Hopf algebra $L$ in $Z(A\md\mod)$ and a non-degenerate Hopf algebra pairing $\langle \cdot,\cdot\rangle:K \otimes L \to k$ in the category $Z(A\md\mod)$, we can construct a canonical braided equivalence by Thm 3.20 in \cite{BLS15}: 
$$\Omega^{\langle \cdot , \cdot  \rangle}:Z(A\ltimes K\md\mod) \ito Z(A \ltimes L\md\mod)$$ 
In the special case of $L:=K$, the functor $\Omega^{\langle \cdot , \cdot  \rangle}$ is a braided autoequivalence of $Z(A\ltimes K\md\mod)=DH\md\mod$. In this case, we identify $\langle \cdot , \cdot \rangle$ canonically with an isomorphism of Hopf algebras in $Z(A\md\mod)$ that we denote by $\delta:K \ito K^*$. We call the triple $(A,K,\delta)$ a partial dualization datum and $r_{A,K,\delta}:=\Omega^{\langle \cdot ,\cdot\rangle} \in \Aut_{br}(DH\md\mod)$ a partial dualization of $H$ on $K$. \\
In the case of a group algebra $H=kG$ of a finite group $G$, we obtain for each decomposition of $G$ as a semi-direct product $G=Q \ltimes N$ a decomposition of $kG$ as a Radford biproduct $kG= kQ \ltimes kN$, where $N$ is a normal subgroup of $G$. $kN$ is a Hopf algebra in $Z(kQ\md\mod)$, where $kQ$ acts on $kN$ by conjugation and where the $kQ$-coaction on $kN$ is trivial. The existence of a Hopf isomorphism $\delta:kN \ito k^N$ in the category $Z(kQ\md\mod)$ forces $N$ to be abelian and $kN$ to be a self-dual $kQ$-module. Thus, for each partial dualization datum $(Q,N,\delta)$, we obtain an element $r_{Q,N,\delta} \in \Aut_{br}(DG\md\mod)$. We denote the set of partial dualizations by $\cR$.

\begin{question}\label{q_decomposition}
	Do the subgroups $\im(\Ind_{H\md\mod})$, $\im(\Ind_{H^*\md\mod})$ together with partial dualizations $\cR$ generate the group $\Aut_{br}(DH\md\mod)$? Does $\Aut_{br}(DH\md\mod)$ decompose as an ordered product of $\im(\Ind_{H\md\mod})$, $\im(\Ind_{H^*\md\mod})$ and $\cR$? 
\end{question}

\noindent
Natural questions for applications are: 

\begin{question}
      The elements of $\im(\Ind_\cat),\im(\Ind_{\cat'})$ are by definition
      realized as different bimodule category structures of the abelian categories
			$\cat$ and $\cat'$ respectively. What are the bimodule categories associated to the partial
      dualizations?
\end{question}

\begin{question}
      What are the three types of group extensions of the fusion category $\cat$
      associated by the isomorphism in \cite{ENO10} to the two subgroups and to partial dualizations?  
\end{question}

A decomposition as described in Question \ref{q_decomposition} would give us effective control over the Brauer-Picard group $\BrPic(\cat)$ through explicit and natural generators. Additionally, these generators have an interesting field theoretic interpretation (see next page). \\

In the present article we consider the case $H=kG$ with $G$ a finite group, hence $\cat=\Vect_G$, $\cat'=\Rep(G)$; in this case the subgroups $\Aut_{mon}(\Vect_G)$ and to a lesser extend $\Aut_{mon}(\Rep(G))$ are well-known. As a main result, we prove that the decomposition described in Question \ref{q_decomposition} holds for the subgroup of elements in $\BrPic(\Vect_G)$,  
which fulfill the additional cohomological property of \emph{laziness}. This condition is automatically fulfilled in the case that $G$ is abelian. Further, for some known examples, we check that the decomposition holds also for the full Brauer-Picard group (see Section \ref{sec_examples}). One important example is the following 

\begin{exampleX}[Sec. \ref{sec_Fp_AutBr}]
Let $G  = \ZZ_p^n$ with $p$ a prime number. Fixing an isomorphism $\ZZ_p \simeq \widehat{\ZZ}_p$, we have the following group isomorphism: $$\BrPic(\Rep(\ZZ_p^n)) \simeq \O_{2n}(\F_p,q)$$ where 
$\O_{2n}(\F_p,q)$ is the group of invertible matrices in $\F_p^{2n\times 2n}$ that are invariant under $q:\F_p \times \F_p \to \F_p;(k_1,...,k_n,l_1,...,l_n) \mapsto \sum_{i=1}^{n}k_il_i$. 
In this case, the images of $\Ind_\cat$ resp.
$\Ind_{\cat'}$ in these Lie groups are lower resp. upper block triangular matrices, intersecting in the subgroup $\Out(\ZZ_p^n) \simeq \GL_n(\F_p)$. The partial
dualizations are Weyl group elements. Our result gives an analogue of the Bruhat decomposition of the Lie
group $\O_{2n}(\F_p,q)$. There are $n+1$ double cosets of the parabolic Weyl group $\SS_n$,
accounting for the $n+1$ non-isomorphic partial dualizations on subgroups $\ZZ_p^k$ for $k=0,...,n$. 
\end{exampleX}

Our general decomposition is modeled after this example and retains roughly what remains of the Bruhat decomposition for a Lie group over a ring (say in the case $G=\ZZ_k^n$ with $k$ not prime), but it is not a Bruhat decomposition in general. Moreover, for $G$ non-abelian the subgroups $\Ind_{\cat},\Ind_{\cat'}$ in $\Aut_{br}(\DG\md\mod)$ are not isomorphic. Additionally, we exhibit a rare class of braided autoequivalences acting as the identity functor on objects and morphisms but having a non-trivial monoidal structure.\\

From a mathematical physics perspective these subgroups arise as follows: A Dijkgraaf-Witten theory has as input data a finite group $G$ and a $3$-cocycle $\omega$ on $G$. It is a topological gauge theory with principal $G$-bundles on a manifold $M$ as classical fields. Since for a finite group $G$ all $G$-bundles are flat, they already form the configuration space. The $\omega$ corresponds to a Lagrangian functional (in our article $\omega$ is trivial). We are interested in the symmetry group of the quantized theory, which is per definition the group of invertible defects and hence $\BrPic(\Vect_G)$.  

Based on this gauge theoretic view, it is natural to expect automorphisms of $G$ to be a symmetry of the classical and the quantized theory. Indeed, $\cV=\Out(G)$ is a subgroup of $\Aut_{br}(\DG\md\mod)$ and since it already exists at the classical level, we call this a classical symmetry (see Proposition \ref{vcat}). It is both, a subgroup of $\im(\Ind_{\Vect_G})$ and a subgroup of $\im(\Ind_{\Rep(G)})$. More symmetries can be obtained by the following idea: equivalence classes of gauge fields are principal $G$-bundles and thus are in bijection with homotopy classes of maps from $M$ to $BG$, the classifying space of $G$. One may view a Dijkgraaf-Witten theory based on $(G,\omega)$ as a $\sigma$-model with target space $BG$. Then the $3$-cocycle $\omega$ can be viewed as a background field on the target space, and the choice of $\omega$ corresponds to the choice of a $2$-gerbe. Even for trivial $\omega$ we obtain a non-trivial symmetry group of this $2$-gerbe and hence an additional subgroup of automorphisms of the theory. These symmetries are again classical symmetries, the so-called \emph{background field symmetries} $\H^2(G,k^\times)$. Our subgroup $\im(\Ind_{\Vect_G})=\cB\rtimes \cV$ where $\cB \cong \H^2(G,k^\times)$ is therefore the semidirect product of the two classical symmetry groups from above (see Proposition \ref{lm_BField}). An interesting implication of our result is that in order to obtain the full automorphism group one considers a second $\sigma$-model (the 'dual $\sigma$-model') associated to $\cat'=\Rep(G)$ which however leads to the same quantum field theory. This dual $\sigma$-model induces another subgroup of background field symmetries $\cE$ which is a subgroup of $\im(\Ind_{\Rep(G)})$ (see Proposition \ref{efield}). However, the group $\im(\Ind_{\Rep(G)})$ is \emph{not} a semidirect product of $\cE$ and $\cV$ in general.

The elements in $\cR$ have the field theoretic interpretation of so-called \emph{partial em-dualities} (electric-magnetic-dualities). In so-called quantum double models (see e.g. \cite{BCKA13} and \cite{KaW07}), irreducible representations of $\DG$ have the interpretation of quasi-particle charges (anyon charges). These are parametrized by pairs $([g],\chi)$, where $[g]$ is a conjugacy class in $G$ and $\chi$ an irreducible representation of the centralizer $\Cent(g)$. The irreducibles of the form $([g],1)$ are called \emph{magnetic} and the ones of the form $([1],\chi)$ are called \emph{electric}. Partial dualizations $\cR$ are symmetries of the quantized theory that exchange magnetic and electric charges (see Proposition \ref{pcat} and Section \ref{sec_nonlazyReflection}). These are only present at the quantum level, hence we call them quantum symmetries. \\ 
For the lazy Brauer-Picard group, which incorporates the abelian case as a special case, it is enough to dualize on direct abelian factors of $G$. A partial dualization is then induced by a Hopf automorphisms of $DG$ (Proposition \ref{pcat}). For the general Brauer-Picard group, we need to consider semi-direct abelian factors of $G$. Partial dualizations are then induced by algebra isomorphisms that are not necessarily Hopf (see Section \ref{sec_nonlazyReflection}). \\ 

\noindent
We now outline the structure of this article and give details on our methods and results:\\

In Section 2, we give some preliminaries: We recall the definition the Drinfeld double $\DG$ and list the irreducible modules $\ocat_g^\chi$ to be able to express our result also in this explicit basis. Further, we give some basic facts about Hopf Bigalois objects, these are certain $H^*$-bicomodule algebras $A$ such that the functor $A \otimes_H \bullet$ gives an element in $\Aut_{mon}(H\md\mod)$ - all monoidal equivalences of $H\md\mod$ arise in this way for some Bigalois object. A special class of Bigalois objects is given by \emph{lazy} Bigalois objects: These are described by pairs $(\phi,\sigma)$ where $\phi \in \Aut_{Hopf}(H\md\mod)$ describes the action of $A \otimes_H \bullet$ on $H$-modules and where $\sigma \in \Z^2_L(H^*)$, a lazy $2$-cocycle, describes a monoidal structure on the functor $A \otimes_H \bullet$. \\

In Section 3, we recall the decomposition of the group of Hopf algebra automorphisms $\Aut_{Hopf}(\DG)$ we have obtained in \cite{LP15} into certain subgroups. These subgroups can be seen as upper triangular matrices $E$, lower triangular matrices $B$, block diagonal matrices $V \cong \Aut(G)$ and $V_c \cong \Aut_c(G)$ and so called \emph{reflections} on direct abelian factors of $G$. \\

In Section 4, we construct certain \emph{braided} lazy autoequivalences of $\DG\md\mod$. For this we first consider lazy monoidal autoequivalences $\Aut_{mon}(\DG\md\mod)$. These can be parametrized by pairs 
$$(\phi,\sigma) \in \Aut_{Hopf}(\DG) \ltimes \Z^2_L(\DG^*)$$
where a pair $(\phi,\sigma)$ corresponds to the functor $(F_\phi,J^{\sigma})$
acting on objects via $\phi$ together with a monoidal structure $J^\sigma$ determined
by the lazy $2$-cocycle $\sigma$. A $2$-cocycle on a Hopf algebra $H$ is lazy if the Doi twist with $\sigma$ gives again the same Hopf algebra structure. We note that two different pairs may of course give functors that are monoidal equivalent. For example, they might differ by a pair consisting of an inner Hopf automorphism and an exact $2$-cocycle. Additionally, internal Hopf automorphisms produce trivial monoidal autoequivalences. This leads us to the following tedious notation:
Let $\underline{\Aut}_{mon}(\DG \md \mod)$ be the category of monoidal autoequivalences and $\underline{\Aut}_{mon,L}(\DG \md \mod) := \Aut_{Hopf}(\DG) \ltimes \Z^2_L(\DG^*)$. Further, let $\widetilde{\Aut}_{mon,L}(\DG \md\mod) := \Out_{Hopf}(\DG) \ltimes \H^2_L(\DG^*)$ be the quotient by lazy coboundaries and inner automorphisms. Let $\Aut_{mon}(\DG \md \mod)$, $\Aut_{mon,L}(\DG \md \mod)$ be the groups of \emph{equivalence classes} of monoidal autoequivalences respectively lazy monoidal autoequivalences. Note that $\widetilde{\Aut}_{mon,L}(\DG \md\mod) \twoheadrightarrow \Aut_{mon,L}(\DG \md \mod)$ has a non-trivial kernel (see Sequence (\ref{monlazy})).
Accordingly, we use the notation $\Aut_{br}(\DG\md\mod)$, $\widetilde{\Aut}_{br,L}(\DG \md\mod)$ etc. for the corresponding subgroups of braided autoequivalences. For $\mathcal{\underline{U}} \subset \underline{\Aut}_{br,L}(\DG\md\mod)$ we denote the images by $\widetilde{\mathcal{U}} \subset \widetilde{\Aut}_{br,L}(\DG\md\mod)$ and $\mathcal{U} \subset \Aut_{br,L}(\DG\md\mod)$. \\  



The goal of Section 4, is to construct certain subgroups of $\Aut_{br,L}(\DG\md\mod)$. Recall the subgroups $V$, $B$, $E$ as well as the subset $R$ in $\Aut_{Hopf}(\DG)$ from Section 3. We observe that for each of the subsets any suitable element $\phi$ can be combined with a specific $2$-cocycle $\sigma$ in $\H^2(G,k^\times)$ resp. $\H^2_L(k^G)$ resp. a pairing such that the pair $(\phi,\sigma)$ becomes braided. We thus define $\cVUnderL$, $\cBUnderL$, $\cEUnderL$, $\cRUnderL$  $\subset \underline{\Aut}_{br,L}(\DG \md \mod)$ in Propositions \ref{vcat}, \ref{lm_BField}, \ref{efield} and \ref{pcat} as follows: \\ 
$\bullet$ The subgroup $\cVUnderL\cong \Aut(G)$ consists of pairs $(\phi,1)$ i.e. functors induced by group automorphisms $\phi \in \Aut(G)$ on objects and a trivial monoidal structure. The images in the quotients $\underline{\Aut}_{br,L}(\DG\md\mod)$ resp. $\Aut_{br,L}(\DG\md\mod)$ are $\cVTildeL \cong \cVL \cong\Out(G)$. \\
$\bullet$ The subgroup $\cEUnderL$ consists of suitable elements $\phi\in E$, each combined with a specific cocycle $\sigma\in \Z^2(k^G)$. 
More precisely, they are constructed in a way that makes the group homomorphism $\cEUnderL\to \Aut_{Hopf}(\DG)$ given by $(\phi,\sigma)\mapsto \phi$ into an isomorphism on
$$\cETildeL\stackrel{\sim}{\longrightarrow} Z(G)\wedge Z(G)\subset E$$
The image $\cEL \subset \Aut_{br,L}(\DG\md\mod)$ corresponds to lazy elements in the image
$$\Ind_{\Rep(G)}:\;\Aut_{mon}(\Rep(G))\to \BrPic(\Rep(G))$$
(up to $\cVL$). Lazy implies here that they arise from $\Aut_{mon}(\Rep(Z(G)))$. \\
$\bullet$ The subgroup $\cBUnderL$ is constructed similarly as $\cEUnderL$. We combine an element
$\phi \in B$ with a special cocycle $\sigma \in \Z^2(G,k^\times)$. Then the image $\cBL \subset \Aut_{br}(\DG\md\mod)$ corresponds
to lazy elements in the image of 
$$\Ind_{\Vect_G}:\;\Aut_{mon}(\Vect_G)\to \BrPic(\Rep(G))$$
(up to $\cVL$). Lazy implies here that they arise from $\Aut_{mon}(\Vect_{G_{ab}})$. In this case $(\phi,\sigma) \mapsto \phi$ is a surjective group homomorphism
\begin{align*}
\cBTildeL \twoheadrightarrow \hat{G}_{ab}\wedge \hat{G}_{ab}\subset B
\end{align*}
which is \emph{not} injective in general. Rather $\cBTildeL$ is a central extension of $\hat{G}_{ab}\wedge \hat{G}_{ab}$ by conjugation invariant \emph{distinguished cohomology classes} of $G$ (c.f. \cite{Higgs87}).  

For $G$ non-abelian, we have hence an interesting class of braided autoequivalences in $\Aut_{br}(\DG\md\mod)$, which are trivial on objects ($\phi=1$) but have nontrivial monoidal structure $J^\sigma$. These seem to be rather rare, in fact the first nontrivial example arises for $G$ a certain non-abelian group of order $p^{9}$, see Example \ref{exm_p9_B}. \\
$\bullet$ The reflections $\cRUnderL=\cRTildeL=\cR$ arise as follows: For every decomposition $G=H\times C$ and Hopf isomorphism $\delta:kC \simeq k^C$, $C$ is an abelian direct factor of $G$. For a triple $(H,C,\delta)$ we consider reflections $r_{H,C,\delta}\in R \subset \Aut_{Hopf}(\DG)$; we say that two reflections are equivalent $r_{H,C,\delta} \sim r'_{H',C',\delta'}$ whenever there exists a group isomorphism $C \simeq C'$. For every $r \in R$ we find the unique $2$-cocycle induced by a pairing $\lambda$ such that the element $(r,\lambda)$ is the braided autoequivalence. We have $\cRUnderL\cong R$. Note however, that contrary to the abelian case, $\cRUnderL$ does not conjugate $\cEUnderL$ and $\cBUnderL$. Also, in order to describe the Brauer-Picard group for the non-lazy case one needs a notion of partial dualizations on \emph{semidirect} products. These do not give lazy elements, unless the semidirect factor is indeed a direct factor. See also the example $G=\SS_3$ in Section \ref{sec_examples} for non-lazy reflections. \\

In Section 5, we finally prove the main result of this article:\\

\begin{theoremX}[\ref{thm_classification}]~\\
(i) Let $G$ be a finite group then for every element $(\phi,\sigma) \in \Aut_{br,L}(\DG\md\mod)$ there exists a $(r,\lambda) \in \cRL$ such that $(\phi,\sigma)$ is in 
\begin{equation*}
(\cVL \ltimes \cBL)\cEL \cdot (r,\lambda) 
\end{equation*}
(ii) Let $G=H \times C$ where $H$ purely non-abelian and $C$ \emph{elementary} abelian. Then $\Aut_{br,L}(\DG\md mod)$ has a double coset decomposition
\begin{equation*} 
\begin{split}
\Aut_{br,L}(\DG\md mod) &= \bigsqcup_{(r,\lambda) \in
\cRL/\sim} \cVL  \cBL \cdot (r,\lambda) \cdot \cVL  \cEL
\end{split}
\end{equation*}
where two partial dualizations $(r_{H,C,\delta},\lambda)$, $(r'_{H',C',\delta'},\lambda')$ are equivalent if and only if there exists a group isomorphism $C \simeq C'$.
\end{theoremX}	

The proof proceeds roughly as follows: Take an arbitrary element $(\phi,\sigma)$, write $\phi\in \Aut_{Hopf}(\DG)$ according to the decomposition obtained in Section 3 and show that the braiding condition implies that the factors of the decomposition of $\phi$ in $V,E,B,R$ can be lifted to $\cVUnderL,\cEUnderL,\cBUnderL,\cRUnderL$. We simplify $(\phi,\sigma)$ by multiplying with these lifts. Finally, we argue that $(\id,\sigma')$ being a braided autoequivalence implies that $\sigma'$ is (up to a natural equivalence) a conjugation invariant distinguished $2$-cocycle on $G$. \\

We close in Section \ref{sec_examples} by comparing our results in examples for $G$ to the (full) Brauer-Picard group obtained in \cite{NR14} and show that the answer to Question \ref{q_decomposition} is positive in these cases.  

\section{Preliminaries}

We will work with a field $k$ that is algebraically closed and has characteristic zero. We denote by $\widehat{G}$ the group of $1$-dimensional characters of $G$. \\

\subsection{Modules over the Drinfeld double}\hskip 0pt \\ 

\noindent
We assume the reader is familiar with Hopf algebras and the representation theory of Hopf algebras to the extend as given in the standard literature e.g. \cite{Kass94}.  Denote by $kG$ the group algebra and by $k^G= kG^*$ the dual of the group algebra. Both have well known Hopf algebra structures. The Hopf algebra $kG$ acts on itself and on $k^G$ by conjugation: $h \triangleright g = hgh^{-1}$ and $h \triangleright e_g = e_{hgh^{-1}}$, where the functions $e_x \in k^G$ defined by $e_x(y) = \delta_{x,y}$ for a basis of $k^G$. We will also use the convention $g^h := h^{-1}gh$ and $^hg := hgh^{-1}$.  
Starting from a finite dimensional Hopf algebra $H$ one can construct the Drinfeld double $DH$ that is $H^{* \mathrm{cop}}\otimes H$ as a coalgebra (see e.g. \cite{Kass94}). Here we will be mainly interested in the Drinfeld double of the Hopf algebra $kG$ for a finite group $G$. We denote the vector space basis of $\DG$ by $\{e_x \times y\}_{x,y \in G}$. Then $\DG$ has the following Hopf algebra structure:  
$$(e_x \times y)(e_{x'} \times y') = e_x(^yx') (e_x \times yy') \qquad \Delta(e_x \times y) = \sum_{x_1x_2=x}(e_{x_1} \times y) \otimes (e_{x_2} \times y)$$
with the unit $1_{\DG} = \sum_{x \in G} (e_x \times 1_G)$, the counit $\epsilon(e_x \times y) = \delta_{x,1_G}$ and the antipode $S(e_x \times y) = e_{y^{-1}x^{-1}y} \times y^{-1}$. 

\noindent
Later we will also use the Hopf algebra $\DG^*$ the dual Hopf algebra of $\DG$ for this we recall that in $\DG^*$ we have $$(x \times e_y)({x'} \times e_{y'}) = (xx' \times e_y*e_{y'}) \qquad \Delta(x \times e_y) = \sum_{y_1y_2=y}(x \times e_{y_1}) \otimes (x^{y_1} \times e_{y_2})$$ 
In the case the group $G=A$ is abelian we have that $DA \simeq k(\hat{A}\times A)$ and $DA^* \simeq k(A\times\hat{A})$ are isomorphic to each other as Hopf algebras. In general there is no Hopf isomorphism from $\DG$ to $\DG^*$.  
   
\noindent
Let us denote the category of left $\DG$-modules by $\DG \md\mod$ and recall that it is a semisimple braided tensor category as follows: \\
$\bullet$ The simple objects of $\DG \md \mod$ are induced modules $\ocat_g^{\rho} := kG \otimes_{k\Cent(g)}V$, where $[g] \subset G$ is a conjugacy class and $\rho:\Cent_G(g) \rightarrow \GL(V)$ an isomorphism class of an irreducible representation of the centralizer of a representative $g \in [g]$. We have the following left $\DG$-action on $\ocat_g^{\rho}$: 
$$(e_h\times t).(y\otimes v):=e_h((ty)g(ty)^{-1})(ty\otimes v)$$
More explicitly: $\ocat_g^\rho$ is a $G$-graded vector space consisting of $|[g]|$ copies of $V$:  
$$\ocat_g^\rho:=\bigoplus_{ \gamma \in [g]} V_{\gamma},\quad V_{\gamma}:=V$$
Then the action of an element $(e_h\times 1)\in \DG$ is given by projecting to the homogeneous component $V_h$. Choose a set of coset representatives $\{s_i \in G \}$ of $G/\Cent_G(g) \simeq [g]$. Then for a homogenous component $V_\gamma$ with $\gamma \in [g]$ there is a unique representative $s_i \in G$ that corresponds under the conjugation action to the element $\gamma \in [g]$, thus $s_igs_i^{-1}=\gamma$. For an $h \in G$, there is a unique representative $s_j$ such that $hs_i \in s_j\Cent_G(G)$. The action of an element $(1\times h) \in \DG$ is then given by 
\begin{align}\label{dgd}
V_{\gamma}&\to V_{h\gamma h^{-1}}; v\mapsto (1\times h).v:= h.v := \rho(s_j^{-1}hs_i)(v) 
\end{align}
this is indeed well-defined, since $s_jgs_j^{-1}=h \gamma h^{-1}$.  \\
$\bullet$ The monoidal structure on $\DG \md \mod$ is given by the tensor product of $\DG$-modules, i.e. with the diagonal action on the tensor product. \\
$\bullet$ The braiding $\{c_{M,N}:M \otimes N \stackrel{\sim}{\to} N \otimes M \mid M,N \in \DG \md \mod\}$ on $\DG \md \mod$ is defined by the universal $R$-matrix 
$$R = \sum_{g \in G}(e_g \times 1)\otimes (1 \times g) = R_1 \otimes R_2 \in \DG \otimes \DG$$ 
\begin{align}\label{DGbraid} c_{M,N}(m \otimes n) = \tau(R.(m\otimes n)) = R_2.n \otimes R_1.m \end{align}
where $\tau:M \otimes N \to N \otimes M; m \otimes n \mapsto n \otimes m$ is the twist. \\

In the above convention we leave out the sum for $R=R_1 \otimes R_2 \in \DG \otimes \DG$. This should not be confused with the Sweedler notation for the coproduct or coaction. Note that $\DG \md\mod$ is equivalent as braided monoidal category to the category of $G$-Yetter-Drinfeld-modules and the Drinfeld center of the category of $G$-graded vector spaces $Z(\Vect_G)$. \\

Given a monoidal category $\cat$, denote by $\underline{\Aut}_{mon}(\cat)$ be the category of monoidal autoequivalences and natural monoidal transformations. Denote by ${\Aut_{mon}}(\cat)$ the group of isomorphism classes of monoidal autoequivalences. Similarly, for a braided category $\bcat$, let $\underline{\Aut}_{br}(\bcat)$ be the category of braided autoequivalences and natural monoidal transformations and let ${\Aut_{br}}(\bcat)$ be the group of isomorphism classes of braided monoidal autoequivalences. \\

\subsection{Hopf-Galois-Extensions}~\\

\noindent
In order to study braided automorphisms of $\DG\md\mod$ we will make use of the theory of Hopf-Galois extensions. For this our main source is \cite{Schau91}, \cite{Schau96}, \cite{Schau02} and \cite {BC04}. The motivation for this approach lies mainly in the relationship between Galois extensions and monoidal functors as formulated in e.g. in \cite{Schau91} and also stated in Proposition \ref{eqn:monoidal1}. Namely, monoidal functors between the category of $L$-comodules and the category of $H$-comodules are in one-to-one correspondence with $L$-$H$-Bigalois objects. For this reason we are lead to the study of $DG^*$-Bigalois extensions. We have summarized the relevant facts in more detail in Sect.~1 of \cite{LP15}.   


\begin{definition} Let $H$ be a Hopf algebra. A right $H$-comodule algebra $(A,\delta_R)$ is called a right $H$-\emph{Galois extension of } $k$ (or $H$-\emph{Galois object}) if the Galois map 
\begin{diagram} 
\beta_A:&A \otimes A &\rTo^{id_A \otimes \delta_R} &A \otimes A \otimes H &\rTo^{\mu_A \otimes id_H} &A \otimes H \\
&x \otimes y &\rMapsto^{} &x \otimes y_0 \otimes y_1 &\rMapsto^{} &xy_0 \otimes y_1 
\end{diagram} 
is a bijection. A morphism of right $H$-Galois objects is an $H$-colinear algebra morphism. Left $H$-Galois objects are defined similarly. Denote by $\underline{\Gal}$(H) the category of right $H$-Galois extensions of $k$ and by $\mathrm{Gal}(H)$ the set of equivalence classes of right $H$-Galois objects. 
\end{definition}

\begin{definition} Let $L,H$ be two Hopf algebras. An $L$-$H$-\emph{Bigalois object} $A$ is an $L$-$H$-bicomodule algebra which is a left $H$-Galois object and a right $L$-Galois object. Denote by $\Bigal(L,H)$ the set of isomorphism classes of $L$-$H$-Bigalois objects and by $\Bigal(H)$ the set of isomorphism classes of $H$-$H$-Bigalois objects.   
\end{definition}

\noindent
Recall that the \emph{cotensor product} of a right $L$-comodule $(A,\delta_R)$ and a left $L$-comodule $(B,\delta_L)$ is defined by
$$A\square_L B:= \left\{ \sum a\otimes b\in A\otimes B\mid \sum \delta_R(a)\otimes b= \sum a\otimes \delta_L(b)\right\}$$
Moreover, if $A$ is an $E$-$L$-Bigalois object and $B$ an $L$-$H$-Bigalois object then the cotensor product $A\square_L B$ is an $E$-$H$-Bigalois object. 


\begin{proposition} The cotensor product gives $\Bigal(H)$ a group structure. The Hopf algebra $H$ with the natural $H$-$H$-Bigalois object structure is the unit in the group $\Bigal(H)$. Further, we can define a groupoid $\underline{\Bigal}$ where the objects are given by Hopf algebras and the morphisms between two Hopf algebras $L$, $H$ are given by elements in $\Bigal(L,H)$. The composition of morphisms is the cotensor product.  
\label{groupoid}
\end{proposition}


Recall that a fiber functor $H \md \mathrm{comod} \rightarrow \mathrm{Vect}_k$ is a $k$-linear, monoidal, exact, faithful functor that preserves colimits. We denote by $\underline{\mathrm{Fun}}_{fib}(H\md\comod,\mathrm{Vect}_k)$ the category of fiber functors and monoidal natural transformations. Given a $H$-comodule $A$, then $A \square_H \bullet:\mathrm{comod} \rightarrow \mathrm{Vect}_k $ is a $k$-linear functor such that lax monoidal structures on this functor are in one-to-one correspondence with algebra structures on $A$ that make $A$ an $H$-comodule algebra. Such a functor is monoidal (not just lax monoidal) if and only if the Galois map $\beta_A$ is a bijection. Moreover we have an equivalence of categories $\underline{\Gal}(H) \simeq \underline{\mathrm{Fun}}_{fib}(H\md \comod,\Vect_k)$ (see \cite{Ulb89}). In order for the image $A \square_H \bullet$ to have a $H$-comodule structure such that it is an equivalence of $H\md\comod$, we need $A$ to be a n $H$-Bigalois object. Moreover we have:          

\begin{proposition}(Sect. 5 \cite{Schau91}) \\ 
Let $H$ be a Hopf algebra then we have the following group isomorphism: 
\begin{equation*}\begin{split} 
\mathrm{Bigal}(H) & \to \Aut_{mon}(H\md\mathrm{comod}) \\ 
A &\mapsto (A \square_H \bullet,J^A) 
\end{split}\end{equation*} 
where the monoidal structure $J^A$ of the functor $ A \square_H \bullet$ is given by 
\begin{equation}
\begin{split} 
J^A_{V,W}:(A \square_H V) \otimes_k (A \square_H W) &\stackrel{\sim}{\rightarrow} A \square_H (V \otimes_k W) \\ 
     \left(\sum x_i \otimes v_i\right)\otimes \left(\sum y_i \otimes w_i\right) &\mapsto \sum x_iy_i \otimes v_i \otimes w_i 
\end{split}
\label{eqn:monoidal1}
\end{equation} 
\label{fib}
\end{proposition}

There is a large class of $H$-Galois extensions that are $H$ as $H$-comodules. Let us from now on use the Sweedler notation: $\Delta_H(h) = h_1 \otimes h_2$. A $k$-linear, convolution invertible map $\sigma: H \otimes H \to k$ such that $\sigma(1,b) = \epsilon(b) = \sigma(b,1)$ is called a $2$-cocycle on $H$ if for all $a,b,c \in H$
\begin{align*}
\sigma(a_1,b_1)\sigma(a_2b_2,c) = \sigma(b_1,c_1)\sigma(a,b_2c_2)
\end{align*}   

We denote by $\Z^2(H)$ the set of $2$-cocycles and by $\Reg^1(H)$ the set of $k$-linear, convolution invertible maps $\eta:H \to k$ such that $\eta(1)=1$. We have a map $\d:\Reg^1(H) \to \Z^2(H)$ defined by $\d\eta = (\eta \otimes \eta)* \eta^{-1} \circ \mu_H$ for $\eta \in \Reg^1(H)$. The $2$-cocycle in the image of $\d$ are called \emph{closed}. The convolution of $2$-cocycles does not yield a $2$-cocycle in general and $\Z^2(H)$ does not form a group. In order to ensure that the convolution is a group structure on $\Z^2(H)$ we need the following: A $2$-cocycle $\sigma$ is called \emph{lazy} if it (convolution)-commutes with the multiplication in $H$: $\sigma * \mu_H = \mu_H * \sigma$. Hence for all $a,b \in H$:  
\begin{align}\label{lazy}
\sigma(a_1,b_1)a_2b_2 = a_1b_1\sigma(a_2,b_2)
\end{align}
The set of lazy $2$-cocycles $\Z^2_L(H)$ does indeed form a group with respect to convolution. Denote by $\Reg_L^1(H)$ the subgroup of those $\eta \in \Reg^1(H)$ that additionally commute with the identity $\eta*\id_H = \id_H*\eta$. Define the \emph{lazy cohomology group} by $$\H_L^2(H):=\Z^2(H)/\d(\Reg^1_L(H))$$ 
It is possible that $\d \eta$ is a lazy $2$-cocycle even if $\eta \in \Reg^1(H)$ is not in $\Reg_L^1(H)$. Such an $\eta$ is called \emph{almost lazy} and the subgroup of such almost lazy maps is denoted by $\Reg^1_{aL}(H) \subset \Reg^1(H)$. \\ 

If $H$ is finite dimensional (or pointed) every $H$-Galois object is of the form ${_\sigma}H$ for some $2$-cocycle $\sigma \in \Z^2(H)$. ${_\sigma}H$ is $H$ as an $H$-comodule and has a twisted algebra structure: $a \cdot_\sigma b = \sigma(a_1,b_1)a_2b_2$ for $a,b \in H$. Further, an right $H$-Galois object $A$ is automatically an $L\md H$-Bigalois object with $L \simeq {_\sigma} H_{\sigma^{-1}}$, which is $H$ as coalgebra and has the twisted algebra structure: $a \cdot_\sigma b = \sigma(a_1,b_1)a_2b_2\sigma^{-1}(a_3,b_3)$ for $a,b \in H$. If $\sigma$ is lazy this implies ${_\sigma} H_{\sigma^{-1}} = H$. We call a right $H$-Galois object ${_\sigma}H$ lazy of $\sigma$ is lazy and we call an $H$-Bigalois object lazy if its lazy as a right $H$-Galois object. Denote by $\Bigal_{lazy}(H)$ the group of lazy $H$-Bigalois objects. \\

We can assign a pair $(\phi,\sigma) \in \Aut_{Hopf}(H) \times \Z^2_L(H^*)$ to an $H^*$-Bigalois object $^{\phi^*}{_\sigma}H^*$ that is $H^*$ as a right $H^*$-comodule, where the left $H^*$-coaction is given by precomposing with the dual $\phi^*:H^* \to H^*$ and where the algebra structure is twisted by $\sigma$ as above. This assignment induces a group homomorphism on $\Out_{Hopf}(H) \times \Z^2_L(H^*)$ which is not surjective, but has $\Bigal_{lazy}(H^*)$ as image. The kernel of this map is $\Int(H)/\Inn(H)$, where $\Int(H)$ are internal Hopf automorphisms (of the form $x \cdot x^{-1}$ for an invertible $x \in H$) and where $\Inn(H)$ are inner Hopf automorphisms (of the form $x \cdot x^{-1}$ for a group-like $x \in H$). Given an $x \cdot x^{-1} \in \Int(H)/\Inn(H)$, the dual $x^*:H^* \to k$ is in $\Reg_{aL}^1(H^*)$ and the pair $(x \cdot x^{-1},\d(x^*))$ is a non-trivial element in $\mathrm{Out}_{Hopf}(H) \ltimes \H^2_{L}(H^*)$. To summarize, we have an exact sequence of groups 
\begin{align}\label{monlazy} 1 \to \Int(H)/\Inn(H) \to \mathrm{Out}_{Hopf}(H) \ltimes \H^2_{L}(H^*) \to \Bigal(H^*) \end{align}            

where $\Bigal_{lazy}(H^*)$ is given by the image of the third map. We use the group isomorphism in Proposition \ref{fib} to define the group of lazy monoidal autoequivalences $\Aut_{mon,L}(H \md\mod)$ by the corresponding subgroup of $\Aut_{mon}(H \md\mod)$ and accordingly $\Aut_{br,L}(H \md\mod):= \Aut_{mon,L}(H \md\mod) \cap \Aut_{br}(H \md\mod)$ 

\section{Decomposition of $\Aut_{Hopf}(\DG)$}\label{sec_cell}\noindent

According to the Sequence \ref{monlazy}, the study of lazy monoidal autoequivalences can essentially be reduced to the study of Hopf automorphisms and the study of the second lazy cohomology group. In this section, we recall the decomposition of $\Aut_{Hopf}(\DG)$ that was proven in \cite{LP15}. According to \cite{Keil} (Theorem 1.1 and Corollary 1.2) we can uniquely represent every Hopf automorphism $\phi$ of $\DG$ by an invertible matrix $\left( \begin{smallmatrix} u & a \\ b & v \end{smallmatrix}\right)$ where $u \in \End_{Hopf}(k^G)$, $b \in \Hom(G,\widehat{G})$, $a \in \Hom(\widehat{Z(G)},Z(G))$ and $v \in \End(G)$ fulfilling certain equations (see \cite{Keil}). Then $\phi(f \times g) = u(f_{(1)})b(g) \times a(f_{(2)})v(g)$  for all $f \in k^G$ and $g \in G$. Further, matrix multiplication, where the multiplication of the entries is composition of maps and addition of the entries is the respective convolution product, corresponds to the composition in $\Aut_{Hopf}(\DG)$.     

\begin{proposition}\cite{Keil}~\\
Let $\{e_g \times h \mid g,h \in G\}$ be the standard basis of $\DG$. There are the following natural subgroups of $\Aut_{Hopf}(\DG)$: \\  
(i) $V := \left\{ e_g \times h \mapsto e_{v(g)}\times v(h) \mid v \in \Aut(G) \right\} \simeq \left\{ \left( \begin{smallmatrix} v^{-1} & 0 \\ 0 & v \end{smallmatrix}\right) \mid v \in \Aut(G) \right\}$ \\
(ii) $B := \left\{ e_g\times h \mapsto b(h)(g)\;e_g \times h \mid b \in \Hom(G_{ab},\widehat{G}_{ab}) \right\} \simeq \left\{ \left( \begin{smallmatrix} 1 & b \\ 0 & 1 \end{smallmatrix}\right) \mid b \in \Hom(G_{ab},\widehat{G}_{ab}) \right\}$ \\
(iii) $E := \left\{  e_g \times h \mapsto  \sum_{g_1g_2 = g}e_{g_1} \times a(e_{g_2})h \mid a \in \Hom(\widehat{Z(G)},Z(G)) \right\}$ \\ \hspace*{1.35cm}$\simeq \left\{ \left( \begin{smallmatrix} 1 & 0 \\ a & 1 \end{smallmatrix}\right) \mid a \in \Hom(\widehat{Z(G)},Z(G))\right\}$ \\
(iv) $V_c := \left\{ e_g\times h \mapsto e_{w(g)}\times h  \mid w \in \Aut_c(G) \right\} \simeq \left\{ \left( \begin{smallmatrix} 1 & 0 \\ 0 & w \end{smallmatrix}\right) \mid w \in \Aut_c(G) \right\}$ where $\Aut_c(G)$ is the group of central automorphisms, hence $w \in \Aut(G)$ with $w(g)g^{-1} \in Z(G) \; \forall g \in G$. \\

We have $V \simeq \Aut(G), B \simeq \Hom(G_{ab},\widehat{G}_{ab}), E \simeq \Hom(\widehat{Z(G)},Z(G)), V_c \simeq \Aut_c(G)$. We can also write $B \simeq \widehat{G}_{ab}\otimes_{\ZZ} \widehat{G}_{ab}$ and $E \simeq Z(G) \otimes_{\ZZ} Z(G)$. \\
\label{auto}
\end{proposition}

\begin{proposition}\cite{LP15} 
Let $R_t$ be the set of all tuples $(H,C,\delta,\nu)$, where $C$ is an abelian subgroup of $G$ and $H$ is a subgroup of $G$, such that $G = H \times C$, $\delta:kC \stackrel{\sim}{\rightarrow} k^C$ a Hopf isomorphism and $\nu:C \rightarrow C$ a nilpotent homomorphism.  \\  
(i) For $(H,C,\delta,\nu)$ we define a \emph{twisted reflection} $r_{H,C,\delta,\nu}:\DG \rightarrow \DG$ of $C$ by: 
\begin{equation*}
\begin{split}
(f_H,f_C) \times (h,c) &\mapsto (f_H,\delta(c)) \times (h, \delta^{-1}(f_C)\nu(c)) 
\end{split}
\end{equation*} 
where $f_H \in k^H$, $f_C \in k^C$, $h \in H$ and $c \in C$. All $r_{H,C,\delta,\nu}$ are Hopf automorphisms. \\
(ii) Denote by $R$ the subset of $R_t$ with elements fulfilling $\nu=1_C$. We call the corresponding Hopf automorphisms \emph{reflections} on $C$.  \\    
\label{reflections}
\end{proposition}

\begin{theorem}~\label{thm_cell}\cite{LP15} \\
(i) Let $G$ be a finite group, then $\Aut_{Hopf}(\DG)$ is generated by the subgroups $V$, $V_c$, $B$, $E$ and the set of reflections $R$. \\
(ii) For every $\phi \in \Aut_{Hopf}(\DG)$ there is a twisted reflection $r = r_{H,C,\delta,\nu} \in R_t$ such that $\phi$ is an element in the double coset 
\begin{equation*} \begin{split}
&[(V_c \rtimes V) \ltimes B] \cdot r \cdot [(V_c \rtimes V) \ltimes E]
\end{split}
\end{equation*}
(iii) Two double cosets corresponding to reflections $(C,H,\delta),(C',H',\delta') \in R$ are equal if and only if $C \simeq C'$. \\  
(iv) For every $\phi \in \Aut_{Hopf}(\DG)$ there is a reflection $r=r_{H,C,\delta} \in R$ such that $\phi$ is an element in $r \cdot [ B ((V_c \rtimes V) \ltimes E)]$ \\
(v) For every $\phi \in \Aut_{Hopf}(\DG)$ there is a reflection $r = r_{H,C,\delta} \in R$ such that $\phi$ is an element in $[((V_c \rtimes V) \ltimes B)E] \cdot r$ 
\end{theorem}

We illustrate the statement of Theorem \ref{thm_cell} on some examples:

\begin{example}\label{exm_Fp_Aut}
	For $G=\ZZ_p^n$ with $p$ a prime number, we fix an isomorphism $\ZZ_p \simeq \widehat{\ZZ}_p$. We then have $\Aut_{Hopf}(\DG) \simeq \GL_{2n}(\F_p)$ 
	The previously defined subgroups are in this case:
	\begin{itemize}
	\item $V \cong \Aut(G) = \GL_n(\F_p)$ and $V_c\ltimes V\cong \GL_n(\F_p)\times \GL_n(\F_p)$
	\item $B\cong \widehat{G}_{ab}\otimes \widehat{G}_{ab}=\F_p^{n \times  n}$ as 
	additive group.
	\item $E\cong Z(G)\otimes Z(G)=\F_p^{n \times n}$ as additive group.
	\end{itemize}
	All reflections are not twisted and can be described as follows: For each dimension $d\in\{0,\ldots, n\}$ there is a unique isomorphism type $C\cong \F_p^d$. The possible subgroups of this type $C\subset G$ are 
	the Grassmannian $\mbox{Gr}(n,d,G)$, the possible $\delta:C \stackrel{\sim}{\to} \hat{C}$ are parametrized by $\GL_d(\F_p)$ and in this fashion $R$ can be enumerated.
	On the other hand, we have only $n+1$ representatives $r_{[C]}$ for each dimension $d$, given for example by permutation matrices.  
One checks this indeed gives a decomposition of $\GL_{2n}(\F_p)$ into $V_cVB$-$V_cVE$-cosets, e.g.
\begin{center}
\begin{tabular}{rrccccc}
$\GL_{4}(\F_p)$ & $=$ & $(V_cVB\cdot r_{[1]}\cdot V_cVE)$ & $\cup$ & $(V_cVB\cdot r_{[\F_p]}\cdot V_cVE)$ & $\cup$ & $(V_cVB\cdot r_{[\F_p^2]}\cdot V_cVE)$ \\
 $|\GL_{4}(\F_p)|$ & $=$ & $p^8|\GL_2(\F_p)|^2$ 
 &$+$ &  $\frac{p^3|\GL_2(\F_p)|^4}{(p-1)^4}$ 
 &$+$ & $p^4|\GL_2(\F_p)|^2$ \\
  & $=$ & $p^8(p^2-1)^2(p^2-p)^2$ 
 & $+$ &  $\frac{p^3(p^2-1)^4(p^2-p)^4}{(p-1)^4}$ 
 & $+$ & $p^4(p^2-1)^2(p^2-p)^2$ \\
  & $=$ & \multicolumn{3}{l}{$(p^4-1)(p^4-p)(p^4-p^2)(p^4-p^3)$} 
 &  & 
\end{tabular}
\end{center}
It corresponds to a decomposition of the Lie algebra $A_{2n-1}$ according to the 
$A_{n-1}\times A_{n-1}$ parabolic subsystem. Especially on the level of Weyl 
groups we have a decomposition as double cosets of the parabolic Weyl group
\begin{align*}
    \SS_{2n}
    &=(\SS_n\times \SS_n)1(\SS_n\times \SS_n)
    \;\cup\; (\SS_n\times \SS_n)(1,1+n)(\SS_n\times \SS_n)\;\cup\\
    & \;\cdots
    \;\cup \;(\SS_n\times \SS_n)(1,1+n)(2,2+n)\cdots (n,2n)(\SS_n\times 
\SS_n)\\
    e.g.\quad |\SS_4|
    &=4+16+4
\end{align*}

In this case, the full Weyl group $\SS_{2n}$ of $\GL_{2n}(\F_p)$ is the set of all 
reflections (as defined above) that preserve a given decomposition 
$G=\F_p\times \cdots\times \F_p$.
\end{example}

\section{Subgroups of $\Aut_{br}(\DG\md\mod)$}
\subsection{General considerations}

Recall from the Preliminaries that every lazy $\DG^*$-Bigalois object is of the form ${_\sigma}^{\phi^*}\DG^*$ for some $\sigma \in \H^2_L(\DG^*)$ and some $\phi \in \Out_{Hopf}(\DG)$. The functor in $\Aut_{mon,L}(\DG\md\mod)$ corresponding to ${_\sigma}^{\phi^*}\DG^*$ under the group isomorphism in Proposition \ref{fib} acts on objects by precomposing with $\phi$. We want to slightly modify this action in order to have nicer formulas. For this we use a anti-automorphism (called flip in Def. 3.1 of \cite{Keil}):  
\begin{align}\label{flip} ^\dag:\Aut_{Hopf}(\DG) \ito \Aut_{Hopf}(\DG); \begin{pmatrix} u & b \\ a & v \end{pmatrix} \mapsto \begin{pmatrix} v^* & b^* \\ a^* & u^* \end{pmatrix} \end{align} where $v^*:k^G \to k^G$ is the dual of $v:kG \to kG$, $u^*:kG \to kG$ is the dual of $u:k^G \to k^G$, and similarly $b^*:kG \to k^G$ and $a^*:k^G \to kG$.

\begin{definition}
We have the following map: 
\begin{align} 
\Psi:\Aut_{Hopf}(\DG) \times \Z_L^2(\DG^*) &\to \underline{\Aut}_{mon}(\DG\md\mod) \nonumber \\
                  (\phi,\sigma) &\mapsto (F_\phi,J^\sigma)
\end{align}
where $F_\phi$ assigns a left $\DG$-module $(M,\rho_L)$ to the left $\DG$-module $(M,\rho_L\circ (\phi^\dag \otimes_k \id))$ simply denoted by $_{\phi}M$. The monoidal structure is given by 
\begin{align*}
J^{\sigma}_{M,N}: {_\phi}M \otimes_k {_\phi}N \stackrel{\sim}{\to} {_\phi}(M \otimes_k N); \;
m \otimes n \mapsto \sigma_1.m \otimes \sigma_2.m 
\end{align*}
where we view the $2$-cocycle $\sigma \in \Z_L^2(\DG^*)$ as $\sigma=\sigma_1 \otimes \sigma_2 \in \DG \otimes_k \DG$ leaving out the sum. This should not be confused with the Sweedler notation for the coproduct.
\end{definition}   

 \begin{definition}\label{defmonlazy}~Let us define: \\ 
(i)  $\underline{\Aut}_{mon,L}(\DG \md \mod) := \im(\Psi) \simeq \Aut_{Hopf}(\DG) \ltimes \Z^2_L(\DG^*)$. \\  
(ii) $\widetilde{\Aut}_{mon,L}(\DG \md\mod) := \Out_{Hopf}(\DG) \ltimes \H^2_L(\DG^*)$. \\
(iii) $\Aut_{mon,L}(\DG \md \mod)$ as the image of 
  $\widetilde{\Aut}_{mon,L}(\DG \md\mod) \rightarrow \Aut_{mon}(\DG \md \mod)$. \\
(iv) $\underline{\Aut}_{br,L}(\DG \md\mod) := \underline{\Aut}_{mon,L}(\DG \md \mod) \cap \underline{\Aut}_{br}(\DG \md \mod)$ \\
(v) $\widetilde{\Aut}_{br,L}(\DG\md\mod)$ as the image of 
$\underline{\Aut}_{br,L}(\DG \md\mod) \to \widetilde{\Aut}_{mon,L}(\DG \md\mod)$ and $\Aut_{br,L}(\DG\md\mod)$ as the image of
$\widetilde{\Aut}_{br,L}(\DG \md\mod) \to \Aut_{br}(\DG\md\mod)$ \\
(vi) For $\mathcal{\underline{U}} \subset \underline{\Aut}_{br,L}(\DG\md\mod)$ we denote the respective images by $\widetilde{\mathcal{U}} \subset \widetilde{\Aut}_{br,L}(\DG\md\mod)$ and $\mathcal{U} \subset \Aut_{br,L}(\DG\md\mod)$.   
\end{definition} 

Before constructing the subgroups mentioned above, we show some general properties. The following Lemma follows essentially from Theorem 9.4 \cite{Keil} but we want state it a bit differently in our notation. 

\begin{lemma}\label{lm_dgaction} 
Let $\phi\in \Aut_{Hopf}(\DG)$ in the form $\phi(f \times g) = u(f_{(1)})b(g) \times a(f_{(2)})v(g)$ for $f \in k^G$ and $g \in G$.   
Then the functor $F_\phi$ maps a $\DG$-module $M$ to the $\DG$-module ${_\phi}M$ with the action: $$(f \times g)._\phi m := \phi^\dag(f \times g).m := (v^*(f_{(1)})b^*(g) \times a^*(f_{(2)})u^*(g)).m$$  $F_\phi$ has the following explicit form on simple $\DG$-modules:  
\begin{equation*}
F_\phi(\ocat^\rho_g) = \ocat^{(\rho \circ u^*)b(g)}_{a(\rho')v(g)}
\end{equation*}
where we denote by $\rho': \Z(G) \rightarrow k^{\times}$ the one-dimensional representation such that the any central element $z\in \Z(G)$ act in $\rho$ by multiplication with the scalar $\rho'(z)$. In particular $\rho|_{\Z(G)}=\dim(\rho)\cdot \rho'$.      
\end{lemma}

The functor $(F_\phi,J^\sigma) \in \underline{\Aut}_{mon,L}(\DG\md\mod)$ is braided if and only if the following diagram commutes: 

\begin{diagram}
&{_\phi}M \otimes {_\phi}N &\rTo^{F_{\phi}(J^\sigma_{M,N})}  & {_\phi}(M \otimes N) \\
&\dTo_{c_{{_\phi}M,{_\phi}N}} & &\dTo_{F_\phi(c_{M,N})} \\ 
&{_\phi}N \otimes {_\phi}M &\rTo_{F_{\phi}(J^\sigma_{N,M})} &{_\phi}(N \otimes M)   
\end{diagram}
for all $M,N \in \DG \md \mod$. This is equivalent to the fact that for all $\DG$-modules $M,N$  
\begin{align}
R_2.\sigma_2.n \otimes R_1.\sigma_1.m &= \sigma_{1}.\phi^*(R_2).n \otimes \sigma_{2}.\phi^*(R_1).m  
\label{eqn:braid}
\end{align}
holds for all $m \in M$ and $n \in N$. Where $R=R_1 \otimes R_2 = \sum_{x \in G} (e_x \times 1) \otimes (1 \times x)$ is the $R$-matrix of $\DG$ and $c$ the braiding in $\DG\md\mod$. As above, we identify $\sigma \in Z^2_L(\DG^*)$ with an element $\sigma=\sigma_1 \otimes \sigma_2 \in \DG \otimes \DG$. \\

From \cite{LP15} Lem. 5.3, Lem. 5.4 and Lem. 5.6 we know three subgroups of $\Z^2_L(\DG^*)$: 
\begin{itemize}\itemsep5pt
\item $\Z^2_{inv}(G,k^\times)$: group $2$-cocycles $\beta \in \Z^2(G,k^\times)$ such that $\beta(g,h)=\beta(g^t,h^t)$ $\forall t,g \in G$ that are trivially extended to $2$-cocycles on $\DG^*$.  
\item $\Z^2_c(k^G)$: $2$-cocycles $\sigma \in \Z^2(k^G)$ such that $\alpha(e_g,e_h)= 0$ if $g$ or $h$ not in $Z(G)$, extended trivially to $\DG^*$.  
\item $\P_c(kG,k^G)\simeq \Hom(G,Z(G))$: central bialgebra pairings $\lambda:kG \times k^G \to k$ resp. group homomorphisms $G \to Z(G)$. These give cocycles on $\DG^*$ as follows: $\sigma_{\lambda}(g \times e_x,h \times e_y) = \lambda(g,e_y)\epsilon(e_x)$.  
\end{itemize}
On the other hand, $\beta_{\sigma}(g,h)=\sigma(g \times 1, h \times 1)$ defines a $2$-cocycle in $\Z^2_{inv}(G,k^\times)$. Further, for $\chi,\rho \in \widehat{G}$ and $\widehat{G}=\Hom(G,k)$ the group of characters, $\alpha_{\sigma}(\chi,\rho) := \sigma(1 \times \chi,1 \times \rho)$ defines a $2$-cocycle in $\Z^2(\widehat{G},k^\times)$. Also, $\lambda_{\sigma}(g,f):=\sigma^{-1}( (g \times 1)_1, 1 \times f_1)\sigma(1 \times f_2, (h \times 1)_2)$ defines a lazy bialgebra pairing in $\P_L(kG,k^G) \simeq \Hom(G,G)$.    

\begin{lemma}\label{lmm}
Let $\phi \in \Aut_{Hopf}(\DG)$ be given as above by $$ \phi(f \times g) = u(f_1)b(g) \times a(f_2)v(g)$$ and $\sigma \in \Z^2_L(\DG^*)$ such that $(F_\phi,J^\sigma)$ is braided then the following equations have to hold for all $\rho,\chi \in \widehat{G}$, $g,h \in G$:  
\begin{align}
\beta_{\sigma}(g,g^{-1}hg) &= \beta_{\sigma}(h,g)b(h)(v(g)) \label{h1} \\
\alpha_{\sigma}(\rho,\chi) &= \alpha_{\sigma}(\chi,\rho)u(\chi)(a(\rho)) \label{hh2}\\ 
\lambda_{\sigma}(h,\chi) &= b(h)(a(\chi)) \label{h3}\\
\rho(g) &= u(\rho)[v(g)]b(g)[a(\rho)] \label{h4}
\end{align} 
\label{nec}
\end{lemma}
\begin{proof} Evaluating equation (\ref{eqn:braid}) we get 
\begin{align*}
&\sum_{g,t,h,d,k \in G}\sigma(g \times e_t,h \times e_d) (1 \times k).(e_h \times d).n \otimes (e_k \times 1).(e_g \times t).m  \\ 
=&\sum_{g,t,h,d,k \in G}\sigma(g \times e_t,h \times e_d) (e_g \times t).\phi^*(1 \times k).n \otimes (e_h \times d).\phi^*(e_k \times 1).m  
\end{align*}
The left hand side is equal to 
\begin{align*}
&\sum_{g,t,h,d \in G}\sigma(g \times e_t,h \times e_d) (e_{ghg^{-1}} \times gd).n \otimes (e_g \times t).m \\
=&\sum_{g,t,h,d \in G}\sigma(g \times e_t,g^{-1}hg \times e_{g^{-1}d}) (e_{h} \times d).n \otimes (e_g \times t).m  \\ 
\end{align*}
and the right hand side is equal to
\begin{align*}
&\sum_{g,k,t,h,d,k_1k_2=k}\hspace{-0.7cm}\sigma(g \times e_t,h \times e_d) (e_g \times t).(b^*(k) \times u^*(k)).n \otimes (e_h \times d).(v^*(e_{k_1}) \times a^*(e_{k_2})).m  \\
=&\hspace{-0.6cm}\sum_{g,t,h,d,w,y,z,x}\hspace{-0.7cm}\sigma(g \times e_t,h \times e_d) a^w_{x}b(y)(v(z)w)(e_g \times t).(e_y \times u^*(v(z)w)).n \otimes (e_h \times d).(e_z \times x).m  \\
=&\hspace{-0.6cm}\sum_{g,t,h,d,w,y,z,x}\hspace{-0.7cm}\sigma(g \times e_t,h \times e_d) a^w_{x}b(y)(v(z)w)(\delta_{g,tyt^{-1}}e_g \times tu^*(v(z)w)).n \otimes (\delta_{h,dzd^{-1}}e_h \times \d x).m  \\
=&\sum_{t,d,y,z,x,w}\sigma(y \times e_t,z \times e_d)b(y)(v(d^{-1}zd)w)a^w_{x}(e_{y} \times tu^*(v(d^{-1}zd)w)).n \otimes (e_{z} \times \d x).m  \\
=&\sum_{t,d,y,h,x,w}\sigma(h \times e_d,g \times e_t)b(h)(v(t^{-1}gt)w)a^w_{x}(e_{h} \times du^*(v(t^{-1}gt)w)).n \otimes (e_{g} \times tx).m  \\
=&\hspace{-1cm}\sum_{  \myover{t,h,g,d,x,w,d'}{  d = d' u^* \circ v(xt^{-1}gtx^{-1}) u^*(w)}} \hspace{-1cm}\sigma(h \times e_{d'},g \times e_{tx^{-1}})b(h)(v(g)w)a^w_{x}(e_{h} \times d).n \otimes (e_{g} \times t).m  
\end{align*}

Here we have used several times that the homomorphism $a$ is supported on $Z(G)$ and that $b$ maps $G$ to the character group $\hat{G}$ which is abelian. 
We now that the above equality of the right and left hand side have to hold in particular for the regular $\DG$-module and the elements $m=n=1$. This implies: 
\begin{align}\label{master}
\sigma(g \times e_t , h^g \times e_{g^{-1}d}) &= \hspace{-1cm} \sum_{\myover{x,w,d'}{ d = d' u^*(w)(u^*\circ v)(g^t)}} \hspace{-1cm} \sigma(h \times e_{d'}, g \times e_{tx^{-1}})b(h)(v(g)w)a^w_x 
\end{align} 
for all $g,h,d,t \in G$ where $a^w_x = e_w(a(e_x))$. On the other hand, if equation (\ref{master}) holds, then also the right and left hand side above are equal. 
Let us set $g=1$, sum over all $d$, multiply with $\chi(t)$ for $\chi \in \widehat{G}$ and sum over $t$ in ($\ref{master}$):  
\begin{align*}
\sigma(1 \times \chi, h \times 1) &= \sum_{x,w,t}\chi(t)\sigma(h \times 1, 1 \times e_{tx^{-1}})b(h)(w)e_w(a(e_x)) \\ 
&= \sum_{x,t}\chi(t)\chi(x) \sigma(h \times 1, 1 \times e_{t})b(h)((a(e_x))) \\ 
&= \sigma(h \times 1, 1 \times \chi) b(h)(a(\chi))   
\end{align*}
applying the convolution with $\sigma^{-1}$ on both sides leads to equation ($\ref{h3}$). Further, we multiply both sides of equation $(\ref{master})$ with $\rho(t),\chi(d)$ for some $\chi,\rho \in \hat{G}$ and sum over all $t,d \in G$:
\begin{align*}
\sigma(g \times \rho, h^g \times \chi)\chi(g) = \sigma(h \times \chi, g \times \rho)\chi(a(\rho)) \chi(u^*\circ v(g))b(h)(v(g)a(\rho))
\end{align*}
Setting $\chi=1=\rho$ gives equation ($\ref{h1}$) and setting $g=1=h$ gives equation (\ref{hh2}). On the other hand setting $g=h$ and $\rho=\chi$ and using equations (\ref{h1}),(\ref{hh2}) we have: $$\sigma(g \times \rho, g \times \rho)\rho(g) = \sigma(g \times \rho, g \times \rho)u(\rho)(v(g))b(g)(a(\rho))$$
This almost implies the last equation (\ref{h4}) but it is not yet clear that $\sigma(g \times \rho, g \times \rho)$ is never zero, since elements of the form $g \times \rho$ are not group-like in $\DG^*$. However, we can argue as follows: 
Apply the $2$-cocycle condition several times  
\begin{align} \label{decsigma}
\sigma&(g \times e_x, h \times e_y) = \sigma((g \times 1)(1 \times e_x), h \times e_y) \nonumber \\
&= \sum_{\myover{x_1x_2x_3=x}{y_1y_2=y}} \sigma^{-1}(g \times 1 ,1 \times e_{x_1})\sigma(1 \times e_{x_2}, (h \times 1) (1 \times e_{y_1}))\sigma(g \times 1, (h \times 1 )(1 \times e_{x_3}e_{y_2 })) \nonumber \\
&= \sum_{\myover{x_1x_2x_3x_4x_5=x}{y_1y_2y_3y_4y_5=y}} \sigma^{-1}(g \times 1, 1 \times e_{x_1} )\sigma^{-1}(1 \times e_{y_1}, h \times 1) \sigma(1 \times e_{x_2}, 1 \times e_{y_2}) \nonumber \\ & \qquad \sigma(1 \times e_{x_3}e_{y_3}, h \times 1)\sigma^{-1}(h \times 1, 1 \times e_{x_4}e_{y_4})\sigma(g \times 1, h \times 1)\sigma(gh \times 1, 1 \times e_{x_5}e_{y_5})
\end{align}
which on characters gives: 
\begin{align*} 
\sigma(g \times \chi, h \times \rho) = \sigma^{-1}(g \times 1, 1 \times \chi )&\alpha_{\sigma}^{-1}(1 \times \rho, h \times 1) \alpha_{\sigma}(\chi,\rho)\lambda_{\sigma}(h,\chi \rho) \\ \cdot &\beta_{\sigma}(g,h)\sigma(gh \times 1, 1 \times \chi \rho) 
\end{align*}
since $\beta_{\sigma} \in \Z^2(G,k^\times)$ and $\alpha_{\sigma} \in \Z^2(\widehat{G},k^\times)$ the only thing left is: 
\begin{align*} 
1&= (\sigma^{-1}*\sigma)(g \times 1, 1 \times \chi) = \sum_{\myover{t \in G}{g^t=g}}\sigma^{-1}(g \times e_t, 1 \times \chi)\sigma(g^t \times 1, 1 \times \chi) \\ 
&= \sigma^{-1}(g \times 1, 1 \times \chi)\sigma(g \times 1, 1 \times \chi)
\end{align*}
Hence elements of the form $\sigma(g \times 1, 1 \times \chi)$ and $\sigma(1 \times \chi, g \times 1)$ are also non zero and it follows that $\sigma(g \times \rho, g \times \rho)$ is also never zero which proves equation (\ref{h4}).  \\
\end{proof}

\subsection{Automorphism Symmetries}~ \\

We have seen in Definition \ref{auto} that a group automorphism $v \in \Aut(G)$ induces a Hopf automorphism in $V \subset \Aut_{Hopf}(\DG)$. We now show that automorphisms of $G$ also naturally induce braided autoequivalences of $\DG\md\mod$. \\  

\begin{proposition}~\\
(i) Consider the subgroup $\cVUnderL:=V\times \{1\}$ of $\Aut_{Hopf}(\DG) \ltimes \Z_L^2(\DG^*)$. For an element $(v,1) \in \cVUnderL$ the corresponding monoidal functor $\Psi(v,1)=(F_v,J^{triv})$ with trivial monoidal structure is given on simple objects by 
$$F_{v}(\ocat^\rho_g)  = \ocat^{v^{-1*}(\rho)}_{v(g)}$$
(ii) Every $\Psi(v,1)$ is braided. \\ 
(iii) Let $\cVTildeL$ be the image of $\cVUnderL$ in $\Out_{Hopf}(\DG) \ltimes \H_L^2(\DG^*)$, then we have $\cVTildeL\cong \Out(G)$.
\label{vcat}
\end{proposition}
\begin{proof} 
(i),(iii) Follows from the above and Lemma \ref{lm_dgaction}. \\
(ii) Consider again equation ($\ref{master}$) in the proof of Lemma \ref{lmm}. An element in $\Aut_{Hopf}(\DG) \times \Z^2_L(\DG^*)$ is braided if and only if equation (\ref{master}) is satisfied. For an element $(v,1)$ it is easy to check that its true.  
(iv) The intersection of $\cVUnderL$ with the kernel $\Inn(G)\ltimes \B^2_L(\DG^*)$ is $\Inn(G)$.
\end{proof}   

\begin{example}
For $G=\F_p^n$ we have $V=\GL_n(\F_p)$. 
\end{example}
\begin{example}
The \emph{extraspecial $p$-group} $p_+^{2n+1}$ is a group of order $p^{2n+1}$ generated by elements $x_i,y_i$ for $i\in\{1,2,\ldots n\}$ and the following relations. In particular $2_+^{2+1}=\DD_4$.  
$$x_i^p=y_i^p=1\qquad[x_i,x_j]=[y_i,y_j]=[x_i,y_j]=1,\;for\;i\neq j\qquad [x_i,y_i]=z\in Z(p_+^{2n+1})$$
Then the inner automorphism group is $\Inn(G)\cong \ZZ_p^{2n}$ and the automorphism group is $\Out(G)=\ZZ_{p-1}\ltimes\Sp_{2n}(\F_p)$ for $p\neq 2$ resp. $\Out(G)=\SO_{2n}(\F_2)$ for $p=2$, see \cite{Win72}. \\
\end{example}


\subsection{B-Symmetries}~\\

Now we want to characterize subgroups of $\Aut_{br}(\DG\md\mod)$ corresponding to the lazy induction $\Aut_{mon}(\Vect_G) \to \Aut_{br,L}(\DG\md\mod)$. One fact we need to understand for this is what trivial braided autoequivalences $(1,\beta)$ coming from $\Vect_G$ look like. If the group is abelian, then $\beta$ has to be cohomologically trivial, which implies that the characterization of such elements is easy. On the other hand, if $G$ is not abelian there are non-trivial cocycles $\beta$ leading to non-trivial braided monoidal functors. For this we need the following:

\begin{definition}
	Let $G$ be a finite group. A cohomology class $[\beta]\in \H^2(G,k^\times)$ is called \emph{distinguished} if one of the following equivalent conditions is fulfilled \cite{Higgs87}: \\
(i) The twisted group ring $k_\beta G$ has the same number of irreducible representations as $kG$. Note that $k_\beta G$ for $[\beta]\neq 1$ has no $1$-dimensional representations. \\
(ii) The centers are of equal dimension $\dim Z(k_\beta G)=\dim Z(kG)$. \\
(iii) All conjugacy classes $[x]\subset G$ are \emph{$\beta$-regular}, i.e. for all $g\in\Cent(x)$ we have $\beta(g,x)=\beta(x,g)$. \\
The conditions are clearly independent of the representing $2$-cocycle $\beta$ and the set of distinguished cohomology classes forms a subgroup $\H^2_{dist}(G)$.
\end{definition}
In fact, nontrivial distinguished classes are quite rare and we give in Example \ref{exm_p9_B} a non-abelian group with $p^9$ elements which admits such a class. \\

In the following Proposition we construct $\cBUnderL$ which should be seen as a subset of the functors $\underline{\Aut}_{br}(\DG \md\mod)$. This is of course a large set and we need to identify certain functors. For this reason, as described in the introduction, we consider the quotient $\cBTildeL$, where we identify pairs that differ by by inner Hopf automorphism and exact cocycles. The main property, as shown below, is that up to certain elements this quotient is isomorphic to the group of alternating homomorphisms $G_{ab} \to G_{ab}$. In order to get $\cB_L \subset \Aut_{br}(\DG \md \mod)$ we need to consider the quotient of $\cBTildeL$ by the kernel of $\cBTildeL \to \Aut_{br}(\DG\md\mod)$.

\begin{proposition}\label{lm_BField}~\\ 
(i) The group $B \times \Z^2_{inv}(G)$ is a subgroup of $\Aut_{Hopf}(\DG) \ltimes \Z_L^2(\DG^*)$. An element $(b,\beta)$ corresponds to the monoidal functor $(F_b,J^\beta)$ given by $ F_{b}(\ocat^\rho_g)  = \ocat^{\rho * b(g)}_{g}$ with monoidal structure
\begin{equation*}
\begin{split}
   \ocat^{\rho * b(g)}_{g} \otimes \ocat^{\chi * b(h)}_{h} &\rightarrow F_b(\ocat^\rho_g \otimes \ocat^\chi_h) \\
           (s_m \otimes v) \otimes (r_n \otimes w)  &\mapsto \beta(g_m,h_n) (s_m \otimes v) \otimes (r_n \otimes w)    
\end{split}
\end{equation*}
where $\{s_m\},\{r_n\} \subset G$ are choices of representatives of $G/\mathrm{Cent}(g)$ and $G/\mathrm{Cent}(h)$ respectively and where $g_m=s_mg s_m^{-1}$,$h_n=r_nh r_n^{-1}$. \\   
(ii) The subgroup $\cBUnderL$ of $B \times \Z_{inv}^2(G)$ defined by  
\begin{equation*}
\begin{split}  
\cBUnderL := \{ (b,\beta) \in B \times \Z_{inv}^2(G) \mid b(g)(h) = \frac{\beta(h,g)}{\beta(g,h)} \quad \forall g,h \in G \} 
\end{split}
\end{equation*}
consists of all elements $(b,\beta) \in B \times \Z_{inv}^2(G)$ such that $\Psi(b,\beta)$ is a braided autoequivalence. \\
(iii) Let $B_{alt} \cong \widehat{G}_{ab}\wedge \widehat{G}_{ab}$ be the subgroup of alternating homomorphisms of $B$, i.e. $b\in\Hom(G_{ab},\widehat{G}_{ab})$ with $b(g)(h) = b(h)(g)^{-1}$. Then the following group homomorphism is well-defined and surjective:   
\begin{equation*} 
\begin{split}
\cBUnderL \rightarrow B_{alt}; \quad (b,\beta) \mapsto b  
\end{split}
\end{equation*}
(iv) Let $\cBTildeL$ be the image of $\cBUnderL$ in $\Out_{Hopf}(\DG) \ltimes \H_L^2(\DG^*)$, then we have a central extension
$$1\to \H^2_{dist,inv}(G)\to \cBTildeL\to B_{alt}\to 1$$ \enlargethispage{0.5cm}
where $\H^2_{dist,inv}(G)$ is the cohomology group of conjugation invariant and distinguished cocycles.
\end{proposition}

\noindent
Before we proceed with the proof, we give some examples: 
\begin{example}\label{exm_Fp_B}
For $G=\F_p^n$ we have $B=\hat{G}\otimes \hat{G} \simeq \F_p^{n\times n}$ the additive group of $n\times n$-matrices and $B_{alt}=\hat{G}\wedge \hat{G} \simeq \F_p^{\binom{n}{2}}$ the additive group of $n\times n$-matrices that are skew-symmetric and have zero diagonal entries. Note that the second condition does not follow from the first for $p=2$. For an abelian group there are no distinguished $2$-cohomology-classes, hence $\cB \simeq \cBTildeL \simeq B_{alt}$. 
\end{example}
\begin{example}\label{exm_D4_B}
For $G=\DD_4=\langle x,y \mid x^2=y^2=(xy)^4=1 \rangle$ we have $G_{ab}=\langle \bar{x},\bar{y}\rangle\cong \ZZ_2^2$, $B=\Hom(G_{ab},\hat{G}_{ab})=\ZZ_2^{2\times 2}$ and $B_{alt}=\{1,b\} \cong \ZZ_2$ with $b(\bar{x})(\bar{y})=b(\bar{y})(\bar{x})=-1$. 
It is known that $\H^2(\DD_4,k^\times)=\ZZ_2=\{[1],[\alpha]\}$ and that the non-trivial $2$-cocycles in the class $[\alpha]$ have a non-trivial restriction to the abelian subgroups $\langle x,z\rangle,\langle y,z\rangle\cong\ZZ_2^2$ of $G$. Especially $[\alpha]$ is not a distinguished $2$-cohomology class. By definition of $\cBUnderL$: 
$$\cBUnderL=\{(1,sym),\;(b,\beta\cdot sym)\}$$
where $\beta$ is the pullback of any nontrivial $2$-cocycle in $G_{ab}$ with
$\beta(x,y)\beta(y,x)^{-1}=-1$ and $sym$ denotes any symmetric $2$-cocycles.
Especially $[\beta]=[1]$ as one checks on the abelian subgroups and thus by
definition
$$\cBTildeL=\{(1,[1]),\;(b,[1])\}\cong \ZZ_2$$
However, these $(1,1)$ and $(b,\beta)$, which are pull-backs of two different braided autoequivalences on $G_{ab}$, give rise to the \emph{same} braided equivalence up to monoidal isomorphisms on $G$. Especially in this case we have a \emph{non-injective} homomorphism.
$$\cBTildeL\to \Aut_{mon}(\DG\md\mod)$$   
\end{example}
More generally for the examples $G=p_+^{2n+1}$ we have $B,B_{alt}$ as for the abelian group $\F_p^{2n}$, but (presumably) all braided autoequivalences in $\cBUnderL(\F_p^{2n})$ pull back to a single trivial braided autoequivalence on $G$. \\ 

It tempting to ask if the kernel of $\cBTildeL\to \Aut_{mon}(\DG\md\mod)$ consist of those $(b,\beta)$ for which $[\beta]=[1]$ i.e. if the remaining non-injectivity is controlled by the non-injectivity of the pullback $\H^2(G_{ab})\to \H^2(G)$. \\ 

We give an example where $\cBTildeL\to B_{alt}$ is not injective. This gives us a new 'kind' of a braided autoequivalence $(1,\beta)$ that would be trivial in the abelian case:

\begin{example}\label{exm_p9_B}
	In \cite{Higgs87} p. 277 a group $G$ of order $p^9$ with $\H^2_{dist}(G)=\ZZ_p$ is constructed as follow: We start with the group $\tilde{G}$ of order $p^{10}$ generated by $x_1,x_2,x_3,x_4$ of order $p$, all commutators $[x_i,x_j],i\neq j$ nontrivial of order 
	$p$ and 
	central. Then $\widetilde{G}$ is a central extension of $G:=\tilde{G}/\langle s\rangle$ where $s:=[x_1,x_2][x_3,x_4]$. This central extension corresponds to a class of distinguished $2$-cocycles $\langle\sigma\rangle=\ZZ_p=\H_{dist}^2(G)=\H^2(G)$. This is a 
	consequence of the fact that $s$ cannot be written as a single commutator. Further, we can find a conjugation invariant representative, because there is a conjugation invariant section $G \to \widetilde{G}$.  
	
	The conjugation invariant distinguished $2$-cocycle $\beta$ corresponds to a braided equivalence $(\id,J^\beta)$ trivial on objects. From $G_{ab}\cong \ZZ_p^4$, hence $B_{alt}=\ZZ_p^4\wedge\ZZ_p^4=\ZZ_p^6$ we have a
	 central extension 
	$$1\to \ZZ_p\to \cBTildeL\to \ZZ_p^6\to 1$$ 
	In fact, we assume that the sequence splits and the braided autoequivalence $(\id,J^\beta)$ is the only nontrivial generator of the image $\Psi(\cBTildeL)\subset\Aut_{br}(\DG\md\mod)$, since the pullback $\H^2(G_{ab})\to \H^2(G)$ is trivial. \\
\end{example}

\begin{proof}[Proof of Lemma \ref{lm_BField}]~ 
(i): Let us show that $B$ acts trivially on $\Z^2(G,k^\times)$, hence also on $\Z_{inv}^2(G,k^\times)$: 
\begin{align*}
b.\beta &=  \sum_{x,y,g,h}\epsilon(e_y)\epsilon(e_h)\beta(x,g)(e_x*b^*(y) \times y) \otimes (e_g*b^*(h) \times h)  \\
&= \sum_{x,g} \beta(x,g)(e_x \times 1) \otimes (e_g \times 1)  = \beta
\end{align*}   
For the action on simple $\DG$-modules use Lemma \ref{lm_dgaction}. \\
(ii): Assume $\Psi(b,\beta)$ is braided then according to Lemma \ref{nec} we get for $v=\id$:  
\begin{equation}
b(g)(h) = \beta(h,g)\beta(hgh^{-1},h)^{-1} \quad \forall g,h \in G
\label{eqn:symb1}
\end{equation}
Because $\beta$ is closed we have: $1 = d\beta(h,gh^{-1},h) =\frac{ \beta(gh^{-1},h)\beta(h,g)}{\beta(hgh^{-1},h)\beta(h,gh^{-1})}$ and therefore: 
\begin{align*} 
\quad\quad b(g)(h) = \beta(h,g)\beta(hgh^{-1},h)^{-1} = \beta^{-1}(gh^{-1},h)\beta(h,gh^{-1})
\end{align*}
\begin{align}
\Leftrightarrow b(g)(h)= b(g)(h)b(h)(h) = b(gh)(h) = \beta^{-1}(g,h)\beta(h,g) \label{eqn:symb2}
\end{align}
In the proof of Lemma \ref{nec} we also have shown that $\Psi(b,\beta)$ is braided if and only if (\ref{master}) holds. In this case where $\sigma(g \times e_x, h \times e_y) = \beta(g,h)\epsilon(e_x)\epsilon(e_y)$ it reduces to $(\ref{eqn:symb1})$, hence $\Psi(b,\beta)$ is braided. Since the product of braided autoequivalences is braided this also shows that $\cBUnderL$ is in fact a subgroup of $B \times \Z_{inv}^2(G,k^\times)$. \\ 

(iii) By definition of $\cBUnderL$ we have $b\in B_{alt}$. We now show surjectivity: Let $G_{ab}= G/{[G,G]}$ be the abelianization of $G$ and $\hat{\beta_b} \in Z^2(G_{ab})$ an abelian $2$-cocycle defined uniquely up to cohomology by $b(g)(h) = \hat{\beta_b}(h,g)\hat{\beta_b}(hgh^{-1},h)^{-1} = \hat{\beta_b}(h,g)\hat{\beta_b}(g,h)^{-1}$ for $g,h \in G_{ab}$. Further, we have a canonical surjective homomorphism $\iota:G \rightarrow G_{ab}$ which induces a pullback $\iota^*:Z^2(G_{ab}) \rightarrow \Z_{inv}^2(G,k^\times)$, hence we define $\beta_b := \iota^*\hat{\beta_b}$. \\  

(iv) By (iii) the map $(b,\beta)\mapsto b$ is a group homomorphism $\cBUnderL\to B_{alt}$ and this factorizes to a group homomorphism $\cBTildeL\to B_{alt}$, since $(\Inn(G)\times \B^2(G)) \cap (B\times \Z_{inv}^2(G,k^\times))=1$. The kernel of this homomorphism consists of all $(1,[\beta])\in \cBTildeL$, hence all $(1,[\beta])$ where $[\beta]$ has at least one representative $\beta$ with $\beta(g,x)=\beta(gxg^{-1},g)$ for all $g,x\in G$. We denote this kernel by $K$ and note that it is central in $\cBUnderL$. 

It remains to show $K=\H^2_{dist}(G)$: Whenever $[\beta]\in K$ then there exists a representative $\beta$ with $\beta(g,x)=\beta(gxg^{-1},g)$ for all $g,x\in G$, in particular for any elements $g\in \Cent(x)$, which implies any conjugacy class $[x]$ is $\beta$-regular and thus $[\beta]\in \H^2_{dist}(G)$. For the other direction we need a specific choice of representative: Suppose $[\beta]\in \H^2_{dist}(G)$ and thus all $x$ are $\beta$-regular; by \cite{Higgs87} Lm. 2.1(i) there exists a representative $\beta$ with 
$$\frac{\beta(g,x)\beta(gx,g^{-1})}{\beta(g,g^{-1})}=1$$
for all $\beta$-regular $x$ (i.e. here all $x$) and all $g$. An easy cohomology calculation shows indeed
$$\frac{\beta(g,x)}{\beta(gxg^{-1},g)}
=\frac{\beta(g,x)}{\beta(gxg^{-1},g)}
\cdot \frac{\beta(gx,g^{-1})\beta(gxg^{-1},g)}{\beta(gx,1)\beta(g,g^{-1})}=1$$
hence $(1,\beta)\in \cBUnderL$ by equation (\ref{eqn:symb1}). \\
\end{proof} 

\subsection{E-Symmetries}~\\

It is now natural to construct a subgroup of $E \times \Z^2_c(k^G)$ in a similar fashion. This construction corresponds to the lazy induction $\Aut_{mon}(\Rep(G)) \to \Aut_{br}(\DG\md\mod)$. Unlike in the case of $B$-Symmetries, we do not need to consider some sort of distinguished cocycles. As we will see, being braided for elements of the form $(1,\alpha)$ already implies that the corresponding braided functor is trivial (up to equivalence). \\

In the following Proposition we construct $\cEUnderL$ which, as in the case of $B$-Symmetries, should be thought of as a collection of functors with no equivalence relation. Identifying pairs that differ by inner Hopf automorphisms and exact cocycles gives us $\cETildeL$. As shown below, the main statement is that this quotient is isomorphic to the group of alternating homomorphisms $Z(G) \to Z(G)$. In order to get the subgroup $\cEL \subset \Aut_{br}(\DG\md\mod)$ we have to take the quotient of $\cETildeL$ by the kernel $\cETildeL \to \Aut_{br}(\DG\md\mod)$.   \\    

\begin{proposition}~\\
(i) The group $E \times \Z^2_c(k^G)$ is a subgroup of $\Aut_{Hopf}(\DG) \ltimes \Z^2_L(\DG^*)$. An element $(a,\alpha)$ corresponds to the monoidal functor $\Psi(a,\alpha)=(F_a,J^\alpha)$ given on simple objects by $F_{a}(\ocat^\rho_g)  = \ocat^{\rho}_{a(\rho')g}$, with the monoidal structure
\begin{align*}
\ocat^{\rho}_{a(\rho')g}\otimes \ocat^{\rho}_{a(\chi')h}
&\to F_a(\ocat^\rho_g\otimes \ocat^\chi_h) \\
(s_m \otimes v)\otimes(r_n \otimes w) 
& \mapsto \sum_{\myover{i,j; x \in \mathrm{Cent}(g)}{ y \in \mathrm{Cent}(h)}} \alpha(e_{s_i x s_m^{-1}},e_{r_jyr_n^{-1}}) [s_i \otimes \rho(x)(v)] \otimes [r_j \otimes \chi(y)(w)]    
\end{align*}
where $\rho',\chi': Z(G) \rightarrow k^{*}$ are the $1$-dimensional characters determined by the restrictions of $\rho,\chi$ to $Z(G)$. By $\{s_i\},\{r_j\} \subset G$ we denote the choices of representatives of $G/\mathrm{Cent}(g)$ and $G/\mathrm{Cent}(h)$ respectively.
(ii) The subgroup $\cEUnderL\subset E \times Z^2_c(k^G)$ defined by  
\begin{align*}
\cEUnderL := \{ (a,\alpha) \in E \times & \Z^2_c(k^G) \mid \forall g,t,h \in G: 
\alpha(e_t,e_{ght}) = \alpha(e_t,e_{hg^{-1}t}) \\
&\alpha(e_t,e_h) = \sum_{x,y \in Z(G)} \alpha(e_{hy^{-1}},e_{tx^{-1}})e_y(a(e_x)) \}
\end{align*} 
consists of all elements $(a,\alpha) \in E \times \Z^2_c(k^G)$ such that the monoidal autoequivalence $\Psi(a,\alpha)$ is braided. \\
(iii) Let $E_{alt}\cong Z(G)\wedge_{\ZZ} Z(G)$ be the subgroup of alternating homomorphisms in $E=\Hom(\widehat{Z(G)},Z(G))=Z(G)\otimes_{\ZZ} Z(G)$, i.e. the set of homomorphisms $a:\widehat{Z(G)}\to Z(G)$ with $\rho(a(\chi)) = \chi(a(\rho))^{-1}$ and $\chi(a(\chi))$ for all $\chi,\rho \in \widehat{Z(G)}$. Then the following group homomorphism is well-defined and surjective:
\begin{equation*} 
\begin{split}
\cEUnderL &\rightarrow E_{alt},\quad (a,\alpha) \mapsto a  
\end{split}
\end{equation*} 
(iv) Let $\cETildeL$ be the image of $\cEUnderL$ in $\Out_{Hopf}(\DG)
\ltimes \H_L^2(\DG^*)$, then the previous group homomorphism factorizes to an
isomorphism
$$\cETildeL\cong E_{alt}$$
For each $a\in E_{alt}$ we have a representative functor $\Psi(a,\alpha)=(F_a,J^\alpha)$ for a certain $\alpha$ obtained by pull-back from the center of $G$. More precisely, the functor is given by $F_{a}(\ocat^\rho_g)  = \ocat^{\rho}_{a(\rho')g}$ and the monoidal structure given by a scalar
\begin{align*}
\ocat^{\rho}_{a(\rho')g}\otimes \ocat^{\rho}_{a(\chi')h}
&\to F_a(\ocat^\rho_g\otimes \ocat^\chi_h) \\
m\otimes n &\mapsto \alpha'(\rho',\chi')\cdot (m\otimes n)  
\end{align*}
where $\alpha'\in \Z^2(\widehat{Z(G)})$ is any $2$-cocycle in the cohomology class associated to the alternating form $a\in E_{alt}$ on the abelian group $\widehat{Z(G)}$.
\label{efield}
\end{proposition}

Before we proceed to the proof we give some examples:

\begin{example}\label{exm_D4_E}
For $G=\DD_4=\langle x,y\rangle,x^2=y^2=(xy)^4=1$ we have $Z(G)=\langle [x,y]\rangle\cong \ZZ_2$ and hence $E=\Hom(\widehat{Z(G)},Z(G))=\ZZ_2$ and $E_{alt}=1$. More generally for the examples $G=p_+^{2n+1}$ we have $E=\ZZ_p\otimes\ZZ_p=\ZZ_p$ and $E_{alt}=\ZZ_p\wedge\ZZ_p=1$ and hence $\cETildeL=1$.
\end{example}
\begin{example}\label{exm_p9_P}
	For the group of order $p^9$ in Example \ref{exm_p9_B} we have $Z(G)=\ZZ_p^5$ generated by all commutators 
	$[x_i,x_j],i\neq j$ modulo the relation $[x_1,x_2][x_3,x_4]$. Hence $E_{alt}=\ZZ_p^{5}\wedge \ZZ_p^5\cong \ZZ_p^{\binom{5}{2}}=\ZZ_p^{10}$ and 
	respectively $\cETildeL=\ZZ_p^{10}$.
\end{example}
 

\begin{proof}[Proof of Proposition \ref{efield}]~\\
\noindent	
(i) Let us show that $E$ acts trivially on $Z^2_c(k^G)$. For $a \in E$ and $\alpha \in Z^2_c(k^G)$: 
\begin{align*}
a.\alpha &= \sum_{x_1,x_2,y,z_1,z_2,w}\alpha(e_y,e_w)(e_{x_1} \times a^*(e_{x_2})y) \otimes (e_{z_1} \times a^*(e_{z_2}) w)\\ 
&= \sum_{y,w}\alpha(e_y,e_w)(1 \times y) \otimes (1 \times w) = \alpha
\end{align*}
For the action on simple $\DG$-modules use Lemma \ref{lm_dgaction}. The rest is straight forward. \\
(ii) Let $(a,\alpha) \in E \times \Z^2_c(k^G)$. Then we use again the fact that $\Psi(a,\alpha)$ is braided if and only if equation (\ref{master}) holds. In this case we have $\sigma(g \times e_x, h \times e_y)=\alpha(e_x,e_y)$ and $\Psi(a,\alpha)$ if braided iff for all $g,t,d \in G:$ 
\begin{equation} 
\begin{split}
\alpha(e_t,e_{gd}) &= \sum_{h,x \in Z(G)}\alpha(e_{dh^{-1}(t^{-1}g^{-1}t)},e_{tx^{-1}})e_{h}(a(e_x))                 
\end{split}
\label{eqn:eq_monoidalStructure_E}
\end{equation}
Setting $g=1$ gives us the second defining equation of $\cEUnderL$. Further, (\ref{eqn:eq_monoidalStructure_E}) is equivalent to 
\begin{align} 
\alpha(e_t,e_{gdt^{-1}gt}) &= \sum_{h,x \in Z(G)}\alpha(e_{dh^{-1}},e_{tx^{-1}})e_{h}(a(e_x))                 
\label{eq_monoidalStructure_E2}
\end{align} 
\noindent
and therefore: $\alpha(e_t,e_{gdt^{-1}gt}) = \alpha(e_t,e_d)$ which is equivalent to the first defining equation of $\cEUnderL$. Since the product of braided autoequivalences is braided this also shows that $\cEUnderL$ is in fact a subgroup of $E \times \Z^2_c(k^G)$.\\
\noindent
(iii) We first note that by equation (\ref{hh2}) for $u=\id$ we have $a \in E_{alt}$. We now show surjectivity: Since $Z(G)$ is an abelian group there exists a unique (up to cohomology) $2$-cocycle $\alpha \in \H^2(\widehat{Z(G)})$ with can be pulled back to a $2$-cocycle in $\Z^2_c(k^G)$. Then $(a,\alpha)$ is in $\cEUnderL$ which proves surjectivity.  \\
\noindent
(iv) Before we show the isomorphism we obtain the description of the explicit representatives: In (iii) we constructed preimages $(a,\alpha)$ of each $a\in E_{alt}$ by pulling back a $2$-cocycle $\alpha'\in \Z^2(\widehat{Z(G)})$ in the cohomology class associated to $a$. We now apply the explicit formula in (i) and use that $\alpha$ is only nonzero on $e_g,e_h$ with $g,h\in Z(G)$: Hence we have only nonzero summands for $s_m^{-1}s_i\in Z(G)$, hence $i=m$ and similarly $j=n$. Moreover, $\rho,\chi$ reduce on $Z(G)$ to one dimensional representations $\rho',\chi'$. Evaluating the resulting sum we get the asserted form. \\
\noindent
Next we note that the group homomorphism $\cEUnderL\to E_{alt}$ in (iii) factorizes to a group homomorphism $\cETildeL\to E_{alt}$, since $(\Inn(G)\times \B_L^2(k^G)) \cap (E\times \Z_c^2(k^G))=1$. The kernel of this homomorphism consists of all $(1,[\alpha])\in \cETildeL$, i.e. all classes $[\alpha]$ such that there exists a lazy representative $\alpha\in \Z^2_c(k^G)$. Then, by definition of $\cEUnderL$, the following is fulfilled for a pair $(1,\alpha) \in \cEUnderL$:
$$\alpha(e_t,e_{ght}) = \alpha(e_t,e_{hg^{-1}t}) \quad \alpha(e_g,e_h) =\alpha(e_{h},e_{g})$$
\noindent
By \cite{LP15} Cor. 3.5 a symmetric lazy cocycle $\alpha\in\Z^2_c(k^G)$ is already cohomologically trivial.\\
\end{proof}


\subsection{Partial E-M Dualizations}~\\

Recall that $R$ was the set of triples $(H,C,\delta)$ such that $G = H \times C$ and $\delta: kC \to k^C$ a Hopf isomorphism. Corresponding to that triple there is unique Hopf automorphism of $DG$ that we called $r_{H,C,\delta}$ that exchanges the $C$ and $\hat{C}$. We will identify the triple $(H,C,\delta)$ with the corresponding automorphism $r=r_{H,C,\delta}$ and the other way around. \\

\begin{proposition}~\\
(i) Consider the subset $R\times \P_c(kG,k^G)$ in $\Aut_{Hopf}(\DG)\ltimes \Z^2_L(\DG^*)$. An element $(r,\lambda)$ corresponds to the monoidal functor $\Psi(r,\lambda)=(F_r,J^\lambda)$ given on simple objects by $F_{r}(\ocat^{\rho_H\rho_C}_{hc})  = \ocat^{\rho_H\delta(c)}_{\delta^{-1}(\rho_C)h}$, where we decompose any group element and representation according to the choice $G= H \times C$ into $h\in H,c\in C$ resp. $\rho_H\in \Cent_H(h)\md\mod,\rho_C\in \Cent_C(c)\md\mod$. The monoidal structure $J^\lambda$ is given by
\begin{equation*}
\begin{split}
   \ocat^{\rho_H\delta(c)}_{\delta^{-1}(\rho_C)h} \otimes \ocat^{\chi_H\delta(c')}_{\delta^{-1}(\chi_C)h'} &\rightarrow F_r(\ocat^{\rho_H\rho_C}_{hc} \otimes \ocat^{\chi_H\chi_C}_{h'c'}) \\
           (s_m \otimes v) \otimes (r_n \otimes w)  &\mapsto \sum_{\myover{i }{ z \in \mathrm{Cent}(hc)}}\lambda((h'c')_n,e_{s_izs_m^{-1}}) [s_i \otimes \rho(z)(v)] \otimes (r_n \otimes w)    
\end{split}
\end{equation*}
where $\{s_m\},\{r_n\} \subset G$ are choices of representatives of $G/\mathrm{Cent}(g)$ and $G/\mathrm{Cent}(h)$ respectively and where $(h'c')_n=r_nh'c'r_n^{-1}$. \\ 
(ii) Define the following set uniquely parametrized by decompositions $G=H\times C$:
\begin{align*}  
\cRUnderL := \{ &(r_{H,C,\delta},\lambda) \in R \times \P_c(kG,k^G) \mid \lambda(hc,e_{h'c'}) = \delta_{c,c'}\epsilon(h)\epsilon(e_{h'}) \} 
\end{align*}
Then $\Psi(r_{H,C,\delta},\lambda)$ is a braided autoequivalence iff $(r_{H,C,\delta},\lambda) \in \cRUnderL$. \\  
(iii) For $(r_{H,C,\delta},\lambda) \in \cRUnderL$ the monoidal structure of $\Psi(r_{H,C,\delta},\lambda)$ simplifies:
\begin{align*}
   \ocat^{\rho_H\delta(c)}_{\delta^{-1}(\rho_C)h} \otimes \ocat^{\chi_H\delta(c')}_{\delta^{-1}(\chi_C)h'} &\rightarrow F_r(\ocat^{\rho_H\rho_C}_{hc} \otimes \ocat^{\chi_H\chi_C}_{h'c'}) \\
           m\otimes n  &\mapsto \rho_C(c')\cdot(m\otimes n)    
\end{align*}
\label{pcat}
\end{proposition}
\begin{proof}
(i) For the action on simple $\DG$-modules use Lemma \ref{lm_dgaction}.\\
\noindent
(ii) For $(r_{H,C,\delta},\lambda) \in R \times \P_c(kG,k^G)$ the functor $\Psi(r,\lambda)$ is braided iff the equation (\ref{master}) holds, where we have to consider the case $\sigma(g \times e_x, h \times e_y) = \lambda(g,e_y)\epsilon(e_x)$. Let us denote an element in the group  $G = H \times C$ by $g = g_Hg_C$ and we write $p_C,p_H$ for the obvious projections. Then we check equation (\ref{master}) in this case: 
\begin{equation*}
\begin{split} 
\sum_{x,y,z \in G}\lambda(y^{-1}xy,e_{z})(e_x \times y) \otimes (e_y \times z) = \sum_{x,y,z,w \in G}\delta_{y,w}\lambda(y^{-1}xy,e_{z})(e_x \times y) \otimes (e_w \times z)  
\end{split}
\end{equation*}
has to be equal to  
\begin{align*}
&\sum_{w,y,g_1,g_2}\lambda(w,e_y)(1 \times y)(\delta^*((g_1g_2)_C) \times (g_1g_2)_H) \otimes (e_w \times 1)(e_{g_1}\circ p_H \times \delta^{-1*}(e_{g_2}\circ p_C)) \\
&= \hspace{-0.5cm}\sum_{x,y,g_1,g_2,w,z}\lambda(w,e_y)(1 \times y)\delta^*((g_1g_2)_C)(r)e_z(\delta^{-1*}(e_{g_2}\circ p_C))( e_x \times (g_1g_2)_H) \otimes (e_w \times 1)(e_{g_1}\circ p_H \times z) \\
&= \hspace{-0.5cm}\sum_{\myover{x,y,g_1,g_2,w,z }{ (g_2)_H=1}}\delta_{z_H,1}\lambda(w,e_y)\delta(x_C)((g_1g_2)_C)e_{(g_2)_C}(\delta^{-1}(e_{z_C}))(1 \times y)( e_{x} \times (g_1g_2)_H) \otimes (e_w \times 1)(e_{g_1} \circ p_H \times z) \\
&= \hspace{-0.7cm} \sum_{\myover{w,y,g_1,g_2 }{ (g_2)_H=1, w_H = (g_1)_H}}\hspace{-0.7cm} \delta_{z_H,1}\lambda(w,e_y)\delta(x_C)((g_2)_C)e_{(g_2)\circ p_C}(\delta^{-1}(e_{z_C}))( e_{yxy^{-1}} \times y(g_1)_H) \otimes (e_w  \times z) \\
&= \sum_{x,y,w,z} \delta_{z_H,1}\lambda(w,e_{yw_H^{-1}})\delta(x_C)((\delta^{-1}(e_{z_C})))(e_{x} \times y) \otimes (e_{w} \times z)
\end{align*}
This is equivalent to saying that for all $x,y,w,z \in G$ the following holds:  
\begin{align*}
\delta_{y,w}\lambda(y^{-1}xy,e_{z}) = \delta_{z_H,1}\lambda(w,e_{yw_H^{-1}})\delta(x_C)((\delta^{-1}(e_{z_C})))
\label{pcat}
\end{align*}
So we see that $(r,\lambda)$ fulfills this equation if and only if $\lambda(hc,e_{h'c'}) = \delta(c)(\delta^{-1}(e_{c'}))\epsilon(h)\epsilon(e_{h'}) = \delta_{c,c'}\epsilon(h)\epsilon(e_{h'})$ for all $hc,h'c' \in H \times C$ which is equivalent to the defining equations of $\cRUnderL$. \\

\noindent
(iii) This is a simple calculation using that $C$ is abelian and then that $\lambda(hc,e_{h'c'}) = \delta_{c,c'}\epsilon(h)\epsilon(e_{h'})$ implies $i=m$ and only leaves the term $\delta_{c',z}$.
\end{proof}

\section{Main Result}\label{mainresult}

Recall that we have defined certain characteristic elements of $\Aut_{br}(\DG \md \mod)$ in the Propositions \ref{vcat}, \ref{lm_BField}, \ref{efield}, \ref{pcat} and showed how they can be explicitly calculated: We have that $\cETildeL$ is isomorphic to the group of alternating homomorphisms $\widehat{Z(G)} \to Z(G)$, that $\cBTildeL$ is a central extension of the group of alternating homomorphisms on $G_{ab} \to \widehat{G_{ab}}$ and that $\cRL$ is parametrized by decompositions $G=H \times C$ together with $\delta: kC \simeq k^C$ such that $\delta(c)(\delta^{-1}(e_{c'}))=\delta_{c,c'}$. In our main result we show that these elements generate $\Aut_{br}(\DG \md \mod)$.

\begin{theorem}\label{thm_classification}~\\
(i) Let $G=H \times C$ where $H$ is purely non-abelian and $C$ is \emph{ elementary} abelian. Then the subgroup of $\Aut_{Hopf}(\DG) \ltimes Z^2_L(\DG^*)$ defined by 
\begin{equation*} 
\AutUnder_{br,L}(\DG\md\mod) := \{ (\phi,\sigma) \in \Aut_{Hopf}(\DG) \ltimes \Z^2_L(\DG^*) \mid (F_\phi,J^\sigma) \text{ braided } \}
\end{equation*}
has the following decomposition into disjoint double cosets 
\begin{equation*} 
\begin{split}
\AutUnder_{br,L}(\DG\md mod) &= \bigsqcup_{(r,\lambda) \in
\cRL/\sim} \cVUnderL  \cBUnderL \cdot (r,\lambda)\cdot \d\Reg^1_{aL}(\DG^*) \cdot \cVUnderL  \cEUnderL \; 
\end{split}
\end{equation*}
where two partial dualizations $(r_{H,C,\delta},\lambda)$, $(r'_{H',C',\delta'},\lambda')$ are equivalent if and only if there exists a group isomorphism $C \simeq C'$ and where $\d\Reg^1_{aL}(\DG^*)$ is the group of almost lazy coboundaries on $\DG^*$. \\
Similarly, the quotient $\Aut_{br,L}(\DG\md\mod)$ has a decomposition into double cosets
\begin{equation*} 
\begin{split}
\Aut_{br,L}(\DG\md mod) &= \bigsqcup_{(r,\lambda) \in \cRL/\sim}  \cVL  \cBL
\cdot (r,\lambda) \cdot \cVL  \cEL
\end{split}
\end{equation*} 
(ii) Let $G$ be a finite group with not necessarily elementary abelian direct factors. For every element $(\phi,\sigma) \in \Aut_{br,L}(\DG\md\mod)$ there exists a $(r,\lambda) \in \cRL$ such that $(\phi,\sigma)$ is in 
\begin{equation*}
(r,\lambda) \cdot [ \cBL (\cVL \ltimes \cEL)]
\end{equation*}
and similarly for $\AutUnder_{br,L}(\DG\md\mod)$. \\
(iii) Let $G$ be a finite group with not necessarily elementary abelian direct factors. For every element $(\phi,\sigma) \in \Aut_{br,L}(\DG\md\mod)$ there exists a $(r,\lambda) \in \cRL$ such that $(\phi,\sigma)$ is in 
\begin{equation*}
[(\cVL \ltimes \cBL)\cEL] \cdot (r,\lambda) 
\end{equation*}
and similarly for $\AutUnder_{br,L}(\DG\md\mod)$.
\end{theorem}	

\noindent
Before we turn to the proof, we add some useful facts. \\
(i) $\Psi$ induces a group homomorphism to $\Aut_{br}(\DG\md\mod)$ that factors through
$$\AutUnder_{br,L}(\DG\md\mod)\to \AutTilde_{br,L}(\DG\md\mod)$$
(ii) $\Psi$ is still not necessarily injective, as Example \ref{exm_D4_B} shows. The kernel is controlled by invertible but not group-like elements in $\DG$ (see Cor. 2.23 in \cite{LP15}). \\  
(iii) The group structure of $\AutTilde_{br,L}(\DG\md\mod)$ can be almost
completely read off using he maps from $\cVTildeL,\cBTildeL, \cETildeL,\cRTildeL$
to the known groups (resp. set) $\Out(G), \\ B_{alt},E_{alt},R$ in terms of
matrices. Only $\cBTildeL\to B_{alt}$ is not necessarily a bijection in rare
cases (in these cases additional cohomology calculations are necessary to
determine the group structure). \\
(iv) The decomposition of $\AutUnder_{br,L}(\DG\md\mod)$ is up to up to a monoidal natural transformation which comes from a coboundary in $\d\Reg^1_{aL}(\DG^*)$. \\

\begin{proof}[Proof of Theorem \ref{thm_classification}]~\\
(i) We start with a general element $(\phi,\sigma) \in
\AutUnder_{br,L}(\DG\md\mod)$. As in Theorem \ref{thm_cell} (ii) we
write $\phi$ as a product of elements in $V,V_c,B,E,R$. Since we only have
elementary abelian direct factors the twist $\nu$ is zero. The general procedure
is to multiply the element $(\phi,\sigma)$ with specific elements of
$\cVUnderL,\cBUnderL,\cEUnderL$ in order to simplify the general
form of $\phi$. We will use the symbol $\leadsto$ after an multiplication and warn that
the $u,v,b,a$ before and after the multiplication are in general different.
We will use the matrix notation
with respect to the product $\DG = k^G \rtimes kG$ and also with respect to a
product $G = H \times C$. For example we write an $v \in \Aut(H \times C)$ as
$\left(\begin{smallmatrix} v_{H,H} & v_{C,H} \\ v_{H,C} & v_{C,C} 
\end{smallmatrix} \right)$ and similarly for the $u,b,a$. \\ 
First, it is easy to see that we can find elements $\cVUnderL$ such that $(\phi,\sigma)$ becomes a pair where the automorphism $\phi$ has the form 
\begin{equation}
\leadsto\begin{pmatrix} v^* & 0 \\ 0 & 1\end{pmatrix}\begin{pmatrix} 1 & b \\ 0 & 1\end{pmatrix}\begin{pmatrix} \hat{p}_{H} & \delta \\ \delta^{-1} & p_H \end{pmatrix}\begin{pmatrix} 1 & 0 \\ 0 & w\end{pmatrix}\begin{pmatrix} 1 & 0 \\ a & 1\end{pmatrix}
\label{eqn:st1a}
\end{equation}
and where the $2$-cocycle $\sigma$ stays the same, since the cocycles in $\cVUnderL$ are trivial. Here we used that $V$ normalizes $V_c$ and $E$. Hence with this step we have eliminated the $\cVUnderL \cong \Aut(G)$ parts in $\phi$. Further, we use the fact that the subgroup $\Aut_c(G)$ normalizes the subgroup $B$ and arrive at  
\begin{align}
\leadsto&\begin{pmatrix} 1 & b \\ 0 & 1\end{pmatrix}\begin{pmatrix} v^* & 0 \\ 0 & 1\end{pmatrix}\begin{pmatrix} \hat{p}_{H} & \delta \\ \delta^{-1} & p_H \end{pmatrix}\begin{pmatrix} 1 & 0 \\ 0 & w\end{pmatrix}\begin{pmatrix} 1 & 0 \\ a & 1\end{pmatrix} 
 \label{eqn:st11} \\  = 
&\begin{pmatrix} v^*\hat{p}_{H} + b\delta^{-1}+v^*\delta w a + b p_H w a & v^*\delta w  + bp_Hw \\ \delta^{-1} + p_H w  a & p_H w \end{pmatrix}\label{eqn:st12}
\end{align}
Since (\ref{eqn:st12}) together with the $2$-cocycle $\sigma$ is braided we deduce from Lemma \ref{nec} equation (\ref{h1}) that 
\begin{align} 
1 = [\delta \circ p_C \circ w(g)(v \circ p_H \circ w(g))] \cdot [b \circ  p_H \circ  w(g) (p_H \circ w(g))] 
\label{eqn:step1beta}
\end{align} 
for all $g \in G$. In particular for $g = w^{-1}(h)$ with arbitrary $h \in H$:    
\begin{align} 
1 = b(h)(h) = b_{H,H}(h)(h)   
\end{align}
which implies that $b_{H,H}$ is alternating. Further, taking $g = w^{-1}(h,c)$ in $(\ref{eqn:step1beta})$ we get $\delta(c)(v(h)) = 1$ for all $c \in C, h\in H$, hence $v_{H,C} = 0$. Taking the inverse of (\ref{eqn:st11}) and arguing analogously on the inverse matrix we deduce that $a_{H,H}$ is alternating and that $(w^{-1})_{C,H}=0$ and therefore $w_{C,H} =0$.  
Both such alternating $b_{H,H}$ can be trivially extended to alternating $b = \left(\begin{smallmatrix} b_{H,H} & 0 \\ 0 & 0  \end{smallmatrix} \right)$ on $G$ and similarly for $a_{H,H}$. Now we use Propositions \ref{lm_BField} (iii) and \ref{efield} (iii): For these alternating $a,b$ exist $2$-cocycles $\beta_b \in \Z_{inv}^2(G,k^\times)$ and $\alpha_a \in Z^2_c(k^G)$ such that $(b,\beta_b) \in \cBUnderL$ and $(a,\alpha_a) \in \cEUnderL$. Multiplying equation (\ref{eqn:st11}) with the inverses of $(b,\beta_b)$ and $(a,\alpha_a)$ we simplify equation (\ref{eqn:st11}) to 
\begin{align}
\leadsto&\begin{pmatrix} 1 & b \\ 0 & 1\end{pmatrix}\begin{pmatrix} v^* & 0 \\ 0 & 1\end{pmatrix}\begin{pmatrix} \hat{p}_{H} & \delta \\ \delta^{-1} & p_H \end{pmatrix}\begin{pmatrix} 1 & 0 \\ 0 & w\end{pmatrix}\begin{pmatrix} 1 & 0 \\ a & 1\end{pmatrix} 
 \label{eqn:st2} 
\end{align}
with $ a = \left(\begin{smallmatrix} 0 & a_{C,H} \\ a_{H,C} & a_{C,C}  \end{smallmatrix} \right)$, $b = \left(\begin{smallmatrix} 0 & b_{C,H} \\ b_{H,C} & b_{C,C}  \end{smallmatrix} \right)$, $ v = \left(\begin{smallmatrix} v_{H,H} & v_{C,H} \\ 0 & v_{C,C}  \end{smallmatrix} \right)$ and $ w = \left(\begin{smallmatrix} w_{H,H} & 0 \\ w_{H,C} & w_{C,C}  \end{smallmatrix} \right)$ where the $2$-cocycle $\sigma$ changes to some $2$-cocycle $\sigma'$. The $b$ and $a$ can be simplified even further by using the fact that we can construct alternating $\tilde{b} = \left(\begin{smallmatrix} 0 & \tilde{b}_{C,H} \\ -b_{H,C} & 0  \end{smallmatrix} \right)$ with $\tilde{b}_{C,H}(c)(h) = - 1/ b_{H,C}(h)(c)$ and similarly an alternating $\tilde{a}$. For these maps there exists again $2$-cocycles that lift them to elements in $\cBUnderL$ and $\cEUnderL$ respectively. As before, we multiplying equation (\ref{eqn:st2}) with the inverses of the lifts and get:   

\begin{align}
\leadsto&\begin{pmatrix} 1 & b \\ 0 & 1\end{pmatrix}\begin{pmatrix} v^* & 0 \\ 0 & 1\end{pmatrix}\begin{pmatrix} \hat{p}_{H} & \delta \\ \delta^{-1} & p_H \end{pmatrix}\begin{pmatrix} 1 & 0 \\ 0 & w\end{pmatrix}\begin{pmatrix} 1 & 0 \\ a & 1\end{pmatrix} 
 \label{eqn:st3} 
\end{align}
with $ a = \left(\begin{smallmatrix} 0 & 0 \\ a_{H,C} & a_{C,C}  \end{smallmatrix} \right)$, $b = \left(\begin{smallmatrix} 0 & b_{C,H} \\ 0 & b_{C,C}  \end{smallmatrix} \right)$, $ v = \left(\begin{smallmatrix} v_{H,H} & v_{C,H} \\ 0 & v_{C,C}  \end{smallmatrix} \right)$ and $ w = \left(\begin{smallmatrix} w_{H,H} & 0 \\ w_{H,C} & w_{C,C}  \end{smallmatrix} \right)$ 
Now we commute the matrix corresponding to $b$ to the right as follows:  
\begin{align}
&\begin{pmatrix} 1 & \left(\begin{smallmatrix} 0 & b_{C,H} \\ 0 & b_{C,C}  \end{smallmatrix} \right) \\ 0 & 1\end{pmatrix}  \begin{pmatrix} v^* & 0 \\ 0 & 1\end{pmatrix}\begin{pmatrix} \hat{p}_{H} & \delta \\ \delta^{-1} & p_H+ \end{pmatrix} 
= \begin{pmatrix} v^* & 0 \\ 0 & 1\end{pmatrix}  \begin{pmatrix} 1 & \left(\begin{smallmatrix} 0 & \tilde{b}_{C,H} \\ 0 & \tilde{b}_{C,C}  \end{smallmatrix} \right) \\ 0 & 1\end{pmatrix} \begin{pmatrix} \hat{p}_{H} & \delta \\ \delta^{-1} & p_H \end{pmatrix} \\ 
& = \begin{pmatrix} v^* & 0 \\ 0 & 1\end{pmatrix}  \begin{pmatrix} \hat{p}_{H} & \delta \\ \delta^{-1} & p_H \end{pmatrix} \underbrace{\begin{pmatrix} \left(\begin{smallmatrix} 1 & \tilde{b}_{C,H}\delta^{-1} \\ 0 & 1  \end{smallmatrix} \right) & 0 \\ 0 & 1\end{pmatrix}}_{\in V_c} \underbrace{\begin{pmatrix} 1 & 0 \\ \left(\begin{smallmatrix} 0 & 0 \\ 0 & \delta^{-1} \tilde{b}_{C,C}\delta^{-1}  \end{smallmatrix} \right) & 1\end{pmatrix}}_{\in E} 
\end{align}
By commuting the $V_c$ elements in the decomposition to the right, multiplying with $V$ as in the first step and then commuting back we thus arrived at the following form: 
\begin{align}
\leadsto &\begin{pmatrix} v^* & 0 \\ 0 & 1\end{pmatrix}\begin{pmatrix} \hat{p}_{H} & \delta \\ \delta^{-1} & p_H \end{pmatrix}\begin{pmatrix} 1 & 0 \\ 0 & w\end{pmatrix}\begin{pmatrix} 1 & 0 \\ a & 1\end{pmatrix} 
 \label{eqn:st4} 
\end{align}
with $ a = \left(\begin{smallmatrix} 0 & 0 \\ 0 & a_{C,C}  \end{smallmatrix} \right)$, $ v = \left(\begin{smallmatrix} v_{H,H} & v_{C,H} \\ 0 & v_{C,C}  \end{smallmatrix} \right)$ and $ w = \left(\begin{smallmatrix} w_{H,H} & 0 \\ w_{H,C} & w_{C,C}  \end{smallmatrix} \right)$. Here we eliminated the $a_{H,C}$ part, similarly as the $b_{C,H}$ part, by commuting the corresponding matrix to the left, past trough the reflection. This gives us again an element in $V_c$ which we can absorb. \\
Now consider the inverse of ($\ref{eqn:st4}$):  
$$ \begin{pmatrix} \hat{p}_H (v^*)^{-1} & \delta \\ -a \hat{p}_H (v^*)^{-1} + w^{-1}\delta^{-1}v^{*-1} & -a\delta + w^{-1}p_H \end{pmatrix}$$ is again braided, hence we use as before Lemma \ref{nec} equation (\ref{h1}) to get: 
\begin{align}
1 = \delta(p_C(g))(a( \delta(p_C(g))w^{-1}_{H,C}(p_H(g)))) = \delta(g_C)(a_{C,C}(\delta(g_C))) \delta(g_C)(w^{-1}_{H,C}(g_H)) 
\end{align}
Since this has to hold for all $g=g_Hg_C \in H\times C$ we argue as before and get that $a_{C,C}$ is alternating and that $w_{H,C}^{-1} = 0$ and therefore $w_{H,C}=0$. So we can eliminate the $a_{C,C}$ part by the same arguments as before. Using Lemma \ref{nec} equation (\ref{hh2}) on (\ref{eqn:st4}) we deduce: $v_{C,H} =0$. Since $v$ is diagonal we can commute the matrix to the right through the reflection. We then get a product of a reflection $\delta' = v_{C,C}^* \circ \delta$, $H=H'$ and $v$. In other words, diagonal elements w.r.t a decomposition $G=H\times C$ of $V_c$ normalize reflections of the form $(H,C,\delta)$. We can lift any reflection to an element in $\cRUnderL$ according to Proposition \ref{pcat} (iii). Thus we arrive at: 
\begin{align}
\leadsto &\begin{pmatrix} 1 & 0 \\ 0 & w\end{pmatrix} = \begin{pmatrix} 1 & 0 \\ 0 & \left(\begin{smallmatrix} w_{H,H} & 0 \\0 & w_{C,C}  \end{smallmatrix} \right) \end{pmatrix}
 \label{eqn:st5} 
\end{align}
Applying Lemma \ref{nec} equation (\ref{h4}) on (\ref{eqn:st5}) we get that
$\chi(g) = \chi(w(g))$ for all $g \in G$, hence $w = \id$. During all of the
above multiplications the $2$-cocycle $\sigma$ changed to some other $2$-cocycle $\sigma'$
so that now we are left with only $(1,\sigma')$, which is braided. We want to show that apart from the distinguished $\H^2_{inv}(G)$ part of such a $\sigma'$ such a braided autoequivalence has to be trivial. \\ 
First, we know from \cite{LP15} Lm. 5.3 that $\beta_{\sigma'}(g,h) = \sigma'(g \times 1, h \times 1)$ defines a $2$-cocycle on $G$.  
From equation (\ref{master}) we deduce that if $(1,\sigma')$ is braided then  
\begin{align*}
\sigma'(g \times 1, 1 \times e_x) &= \sigma'(1 \times e_x, g \times 1) \\
\sigma'(g \times 1, h^g \times 1) &= \sigma'(h \times 1, g \times 1) \\
\sigma'(1 \times e_x,1 \times e_y) &= \sigma'(1 \times e_y, 1 \times e_x)
\end{align*}
this shows that $(1,\beta_{\sigma'}) \in \cBUnderL$. We multiply $(1,\sigma')$ from the left with $(1,\sigma_{\beta_{\sigma'}}^{-1})$ where 
$$\sigma_{\beta_{\sigma'}}(g \times e_x, h \times e_y) = \beta_{\sigma'}(g,h)\epsilon(e_x)\epsilon(e_y)=\sigma'(g \times 1, h \times 1)\epsilon(e_x)\epsilon(e_y)$$
and the resulting cocycle fulfills 
\begin{align*} 
\sigma_{\beta_{\sigma'}}^{-1}*\sigma'(g \times 1, h \times 1) &= \sum_{t,s \in G} \sigma^{-1}_{\beta_{\sigma'}}(g \times e_t,h \times e_s)\sigma'(g^t \times 1, h^s \times 1) \\ 
&= \sum_{t,s}\sigma^{-'1}(g \times 1, h \times 1)\epsilon(e_t)\epsilon(e_s)\sigma(g^t \times 1, h^s \times 1) = 1 
\end{align*}
Call the new cocycle again $\sigma'$ and note that it is now trivial if restricted to $kG \times kG$, hence we got rid of the distinguished part of $\sigma$. Further, since $\alpha_{\sigma'}(e_x,e_y) = \sigma'(1 \times e_x, 1 \times e_y)$ is a lazy symmetric $2$-cocycle in $\Z^2_c(k^G)$ it follows from \cite{LP15} Cor. 3.5. that $\alpha_{\sigma'}$ is cohomologically trivial. Let $\eta \in \Reg^1_L(k^G)$ such that $d\eta = \alpha_{\sigma'}$. Use equation \ref{decsigma} from the proof of Lemma \ref{lmm} in this case:
\begin{align*}
\sigma'(g \times e_x, h \times e_y) &= \sum_{\myover{x_1x_2x_3=x}{y_1y_2y_3=y}} \hspace{-0.5cm} \sigma'^{-1}(g \times 1, 1 \times e_{x_1} )\sigma'^{-1}(1 \times e_{y_1}, h \times 1) \d\eta(e_{x_2},e_{y_2})\sigma(gh \times 1, 1 \times e_{x_3}e_{y_3}) \nonumber \\
&=\sum_{\myover{x_1x_2t=x}{y_1y_2t=y}} \sigma'^{-1}(g \times 1, 1 \times e_{x_1} )\sigma'^{-1}(1 \times e_{y_1}, h \times 1) \d\eta(e_{x_2},e_{y_2})\sigma(g^xh^y \times 1, 1 \times e_{t})
\end{align*}
where in the last equation we have used the lazy condition on $\sigma'$. Now let $\mu(g \times e_x) := \sigma^{-1}(g \times 1, 1 \times e_x)$ and check that together with $\eta$ this gives us the desired coboundary to show that $\sigma$ is exact:  
\begin{align*}
&\d(\mu*(\eta \otimes \epsilon_{kG}))(g \times e_x, h \times e_y) \\ &= \sum_{\myover{x_1x_2 =x}{y_1y_2=y}}\mu*(\eta \otimes \epsilon_{kG})(g \times e_{x_1})\mu*(\eta \otimes \epsilon_{kG})(h \times e_{y_1})\mu*(\eta \otimes \epsilon_{kG})(g^{x_1}h^{y_1} \times e_{x_2}e_{y_2}) \\ 
&= \sum_{\myover{x_1x_2x_3x_4 =x}{y_1y_2y_3y_4=y}} \sigma^{-1}(g \times 1, 1 \times e_{x_1})\eta(e_{x_2})\sigma^{-1}(h \times 1, 1 \times e_{y_1})\eta(e_{y_2}) \\ & \qquad \sigma(g^{x_1x_2}h^{y_1y_2} \times 1, 1 \times e_{x_3}e_{y_3})\eta(e_{x_4}e_{y_4}) \\
&= \sum_{\myover{x_1x_2t=x}{y_1y_2t=y}} \sigma^{-1}(g \times 1, 1 \times e_{x_1})\sigma^{-1}(h \times 1, 1 \times e_{y_1})d\eta(e_{x_2},e_{y_2})\sigma(g^{xt^{-1}}h^{yt^{-1}} \times 1, 1 \times e_t) \\ 
&= \sum_{\myover{x_1x_2t=x}{y_1y_2t=y}} \sigma^{-1}(g \times 1, 1 \times e_{x_1})\sigma^{-1}(h \times 1, 1 \times e_{y_1})d\eta(e_{x_2},e_{y_2})\sigma(g^{x}h^{y} \times 1, 1 \times e_t) \\ &= \sigma'(g \times e_x, h \times e_y)
\end{align*}  

\noindent
(ii) By Theorem \ref{thm_cell} (iv) we write
\begin{align}
\phi =  \begin{pmatrix} \hat{p}_{H} & \delta \\ \delta^{-1} & p_H \end{pmatrix} \begin{pmatrix}1 & b \\ 0 & 1\end{pmatrix}\begin{pmatrix} v^* & 0 \\ 0 & 1 \end{pmatrix} \begin{pmatrix} 1 & 0 \\ a & 1\end{pmatrix} 
\label{eqn:lr1} 
\end{align}
where we have already eliminated the $V$ element since it normalizes $E$ and every $V$ has a lift to $\cVUnderL$. Similarly, we know from Proposition \ref{pcat} that (up to an $V$ that ensures $\delta(c)(\delta(e_{c'})) = \delta_{c,c'}$) every reflection $r$ has a lift $(r,\lambda) \in \cRUnderL$. Hence we multiply $(\phi,\sigma)$ with the inverse $(r,\lambda)^{-1}$ from the left so that $\phi$ changes to: 
\begin{align}
\leadsto \begin{pmatrix} 1 & b \\ 0 & 1\end{pmatrix}\begin{pmatrix} v^* & 0 \\ 0 & 1 \end{pmatrix} \begin{pmatrix} 1 & 0 \\ a & 1\end{pmatrix} = \begin{pmatrix} v^* + ba & b \\ a & 1 \end{pmatrix}
\label{eqn:lr2} 
\end{align}
Since this element has to be braided, using Lemma \ref{nec} equation \ref{h1} together with (\ref{eqn:symb2}) it follows that $b$ is alternating on $G$. From Lemma \ref{nec} equation \ref{h4} follows that $v = \id_G$ and then that $a$ is alternating. Hence we can construct lifts to $\cBUnderL$ and $\cEUnderL$ and multiplying with the corresponding inverses just leaves us with a $(1,\sigma')$. As in $(i)$ we get rid of the distinguished part and then this is a trivial autoequivalence (up to natural transformation). \\
\noindent
The proof of (iii) is completely analogous to (ii).    
\end{proof}

\section{Examples and the full Brauer-Picard group}\label{sec_examples}

We now discuss the results of this paper for several classes of groups
$G$. In
particular, we compare our results to the examples obtained in
\cite{NR14}. In
all these cases we verify that the decomposition we proposed in Question
\ref{q_decomposition} is also true
for the full Brauer-Picard group (i.e. the elements do not have to be lazy). \\

The approach in \cite{NR14} is to study $\Aut_{br}(\DG\md\mod)$ via its
action
on the set $\mathbb{L}(G)$ of so-called Lagrangian subcategories $\LL
\subset \DG\md\mod$. These are parametrized by pairs $(N,[\mu])$ where
$N$ is a
normal abelian subgroup of $G$ and $[\mu]$ a $G$-invariant $2$-cohomology
class on $N$. The associated Lagrangian subcategory is
generated, as
abelian category, by simple objects $\ocat_g^\chi$ in the following way (see Sect. 7 \cite{NR14}):
$$\LL_{N,\mu}:=\left\langle \ocat_g^\chi\mid g\in
N,\;\chi(h)=\mu(g,h)\mu(h,g)^{-1}\;\forall{h\in N} \right\rangle$$

Let further $\mathbb{L}_0(G):= \{ \LL \in  \mathbb{L}(G) \mid \LL \simeq \Rep(G) \; \text{as a braided fusion category} \}$ \\
The group $\Aut_{br}(\DG\md\mod)$ acts on the lattice of fusion subcategories of $\DG\md\mod$ and on $\mathbb{L}(G)$. The subset $\mathbb{L}_0(G)$ is invariant under this action. By Prop. 7.6 \cite{NR14}: The action of $\Aut_{br}(\DG\md\mod)$ on $\mathbb{L}_0(G)$ is \emph{transitive}. The stabilizer of the standard Lagrangian subcategory $\LL_{1,1}=\Rep(G)$ is the image of the induction 
$$\Ind_{\Vect_G}:\;\Aut_{mon}(\Vect_G)\to \Aut_{br}(\DG\md\mod)$$
Since the image of $\Ind_{\Vect_G}$ is $\Out(G) \ltimes \H^2(G,k^\times)$, we have a group isomorphism $\mathrm{Stab}(\Rep(G)) \simeq \Out(G) \ltimes \H^2(G,k^\times)$, which implies $|\Aut_{br}(\DG\md\mod)| = |\H^2(G,k^\times)|$ $|\Out(G)|$ $|\mathbb{L}_0(G)|$. \\

We determine our lazy subgroups $\cBL,\cEL,\cRL,\cVL$ for certain examples and show how they act on
$\mathbb{L}$(G). We explicitly calculate the
action in terms of the simple objects $\ocat_g^\chi$. 

\subsection{General considerations on non-lazy
 reflections}\label{sec_nonlazyReflection}~\\

For each triple $(Q,N,\delta)$ where $G$ is a semi-direct product $G=Q \ltimes N$, $N$ a normal abelian subgroup of $G$, $\delta:kN \ito k^N$ a $G$-invariant (under conjugation action) Hopf isomorphism, we obtain an element $r_{Q,N,\delta}:=\Omega \in \Aut_{br}(DG\md\mod)$, where $\Omega$ is the braided autoequivalence given in Thm 3.20 in \cite{BLS15}: We have a decomposition of $kG$ as a Radford biproduct $kG= kQ \ltimes kN$, where $N$ is a normal subgroup of $G$, $kN$ is a Hopf algebra in $DQ\md\mod$, where $kQ$ acts on $kN$ by conjugation and where the $kQ$-coaction on $kN$ is trivial. In our notation, the braided autoequivalence $r_{Q,N,\delta}:\DG\md \mod \ito \DG\md\mod$ assigns a $DG$-module $M$ to a $\DG$-module $r_{Q,N,\delta}(M)$ where $r_{Q,N,\delta}(M)$ is $M$ as a $k$-vector space and where the $\DG$-action on $r_{Q,N,\delta}(M)$ is given by postcomposing with the following algebra isomorphism of $\DG= D(Q \ltimes N)$:   
\begin{align*}
\DG \ni (f_Q,f_N) \times (q,n) &\mapsto (f_Q,\delta(n)) \times (q,\delta^{-1}(f_N)) \in \DG  
\end{align*} 

where $f_Q \in k^Q, f_N \in k^N, q \in Q$ and $n \in N$. Essentially, this is the reflection as defined in Proposition \ref{reflections} but since we do not have a direct product of $Q$ and $N$, we do not have a coalgebra isomorphism as we would have in the lazy case. We are denoting the $\DG$-action on $M$ by $\left((f_Q,f_N) \times (q,n)\right).m \in M$ for $m \in M$ and the $\DG$-action on $r_{Q,N,\delta}(M)$ by $\left((f_Q,f_N) \times (q,n)\right)._r m = \left((f_Q,\delta(n)) \times (q,\delta^{-1}(f_N))\right).m$.   \\   

We show how partial dualization $r_{Q,N,\delta}$ act on irreducibles $\ocat_1^\chi$ and thereby on the subcategory $\LL_{1,1}$. We need this in order to check that the group generated by $\cB,\cE,\cV$ and $\cR$ acts transitively on $\mathbb{L}_0(G)$. Since $\mathbb{L}_0(G)$ is the orbit of $\LL_{1,1}=\Rep(G)$ and since $\LL_{1,1}$ is generated by the simple objects $\ocat^{\chi}_1$, we only need the action of partial dualizations on simple objects of the form $\ocat^{\chi}_1$ for some irreducible character $\chi$ on $G$. \\

Since $r_{Q,N,\delta}$ is an autoequivalence, it sends simple objects to simple objects. Therefore, for each irreducible character $\chi$ on $G$ there exists a conjugacy class $[g] \subset G$ and an irreducible character $\rho$ on $\Cent(g)$ such that $r_{Q,N,\delta}(\ocat_1^\chi)=\ocat_g^\rho$. We have $\dim(\chi)=|[g]|\cdot\dim(\rho)$. We want to determine $[g]$ and $\rho$. \\

Clifford's theorem (see e.g. Page 70, Theorem 4.1 in \cite{Gor07}) states that the restriction of an irreducible character $\chi$ to a normal subgroup $N$ of $G$ 
decomposes into a direct sum of irreducible $N$-characters $\chi_i$ with the same multiplicity $e \in \nat$: 
$$\chi|_N=e\sum_{i=1}^t \chi_i$$
where the $\chi_i$ form a $G$-orbit under conjugation action on $N$ and hence on $\Rep(N)$. The group $Q=G/N$ acts on $\chi_i$ by conjugation in the argument and the subgroups $I_i \subset G/N = Q$ fixing a $\chi_i$ are called the inertia subgroups. We have $[Q:I_i]=t$. \\  
Since $N$ is abelian, we obtain $1$-dimensional representations $\chi_i \in \hat{N}$ forming a $G$-conjugacy class. Then $n_i:=\delta^{-1}(\chi_i) \in N$ are all conjugate to each other in $G$. Fix one representative $n_i=\delta^{-1}(\chi_i)$ in this conjugacy class and the corresponding inertia subgroup $I_i \subset Q$. Further, since $\delta$ is $G$-conjugation invariant we also have the formula $\Cent(n_i) = N \rtimes I_i$. 

We have a decomposition via Clifford's theorem $\ocat^{\chi}_1 = \bigoplus_{j=0}^t T_j \otimes M_j$, where the $M_j$ are $1$-dimensional $k$-vector spaces with an $N$-action given by $\chi_j$ and where $T_j$ is an $e$-dimensional $k$-vector space with trivial $N$-action. Since the partial dualization preserves vector spaces, we also have a decomposition as vector spaces: $r_{Q,N,\delta}(\ocat^\chi_1) = \ocat_g^\rho = \bigoplus_{j=0}^t T_j \otimes M_j$. We calculate the $k^N$-action ($kN$-coaction) on $\ocat_g^\rho$: Let $t \otimes m_j \in T_j \otimes M_j$ then the modified $k^N$-action: 
\begin{align*}
  e_{n_l}._r(t \otimes m_j) & = \delta^{-1}(e_{n_l}).(t \otimes m_j) = \chi_j(\delta^{-1}(e_{n_l})) (t \otimes m_j) \\  
							& = e_{n_l}(\delta^{-1}(\chi_j)) (t \otimes m_j) = e_{n_l}(n_j) (t \otimes m_j) = \delta_{l,j} (t \otimes m_j)
\end{align*}

On the other hand the $k^Q$-action ($kQ$-coaction) on $\ocat_g^{\rho}$ stays the same, which is trivial here. Hence, we have shown: $[g] = [n_i]$. \\ 
Now calculate the action of $\Cent(n_i)=N \rtimes I_i$ on $T_i \otimes M_i$. Let $n \in N$ and note that the $k^G$-action on $\ocat^\chi_1$ is trivial since $|[1]| = 1$. Hence the $kN$-action on $\ocat_{n_i}^{\rho}$ is trivial. For $q \in I_i \subset Q=G/N$: 
$q._r(t \otimes m_i) = (q.t) \otimes m_i$ where $Q$ acts on $T_i$ since $T_i \otimes M_i$ is an $I_i$-submodule. Thus, $\rho$ is the character on $N \rtimes I_i$ which is the trivial extension of the $I_i$-representation $T_i$. Overall we get

$$ r_{Q,N,\delta}(\ocat_1^\chi) = \ocat_{n_i}^{T_i} $$

\subsection{General considerations on non-lazy induction}\label{sec_nonlazyInduction}~\\

We now turn to the subgroups of $\Aut_{br}(\DG\md\mod)$ defined to be the images of the induction 
$$\Ind_{\cat}:\;\Aut_{mon}(\cat)\to\Aut_{br}(Z(\cat))$$
 
for the two cases $\cat =\Vect_G$ and $\cat = \Rep(G)$ where we assign a $F\in\Aut_{mon}(\cat)$ to the invertible $\cat$-bimodule category ${_F}\cat_\cat \in \BrPic(\cat)$ that has a right module category structure given by the monoidal structure of $\cat$ and a left module category structure by composing with $F$ and then using the monoidal structure of $\cat$. We then use the isomorphism $\BrPic(\cat) \simeq \Aut_{br}(Z(\cat))$ to get subgroups of $\Aut_{br}(\DG\md\mod)$. \\

We already know that $\im(\Ind_{\Vect_G}) \simeq \Out(G)\ltimes \H^2(G,k^\times)$. The subgroup $\im(\Ind_{\Rep(G)})$ is harder to compute. The group $\Aut_{mon}(\Rep(G))$ is parametrized by pairs $(N,\alpha)$ where $N$ is an abelian subgroup of $G$ and $\alpha$ belongs to a $G$-invariant cohomology class (see \cite{Dav01}). The subgroup of lazy monoidal autoequivalences corresponds to all pairs where $\alpha$ is $G$-invariant even as an $2$-cocycle. 

\begin{remark}
An interesting example appears when we consider $G=\ZZ_2^{2n}\rtimes \Sp_{2n}(2)$ where $\Sp_{2n}(2)$ is the symplectic group over $\F_2$. There is a pair
$(N,\alpha)$ such that the associated functor is a monoidal equivalence
$$F_{N,\alpha}:\;\Rep(\ZZ_2^{2n}\rtimes \Sp_{2n}(2)) \stackrel{\sim}{\longrightarrow} \Rep(\ZZ_2^{2n}.\Sp_{2n}(2))$$
The groups $\ZZ_2^{2n}\rtimes \Sp_{2n}(2)$ and $\ZZ_2^{2n}.\Sp_{2n}(2)$ are isomorphic only for $n=1$. Namely, they are both isomorphic to $\SS_4$. See Example 7.6 in \cite{Dav01}. This leads to a nontrivial and non-lazy monoidal autoequivalence, which leads to a non-trivial non-lazy braided autoequivalence of $D\SS_4$, see the example below. 
\end{remark}

For any $F\in\Aut_{mon}(\Rep(G))$, we want to determine the image
$$E_{F}:=\Ind_{\Rep(G)}(F)\in \Aut_{br}(\DG\md\mod)$$
Unfortunately, it is not easy to calculate $E_{F}$ explicitly, since it depends on the isomorphism $\BrPic(\cat)\to
\Aut_{br}(Z(\cat))$. In \cite{NR14} equations (16),(17), the image of the induction $\Ind_{\Vect_G}$ was worked out, but we are also interested in the image of $\Ind_{\Rep(G)}$ which seems to be harder. We can easily derive at
least an necessary condition. From \cite{ENO10} we know that given an invertible $\cat$-bimodule category ${_\cat}\mcat_{\cat}$ the corresponding braided autoequivalence $\Phi_\mcat \in \Aut_{br}(Z(\cat))$ is determined by the condition that there exists a isomorphism of $\cat$-bimodule functors $ Z \otimes \cdot \simeq \cdot \otimes \Phi_\mcat(Z)$ for all $Z \in Z(\cat)$. In our case ${_\cat}\mcat_{\cat} = {_F}\cat_\cat$ and $\Phi_\mcat = E_F$. This implies for $(V,c),(V',c') \in Z(\cat)$        
\begin{align*}
E_{F}(V,c)=(V',c')
\;\;\Rightarrow\;\; F(V)\otimes X \cong X\otimes V'\quad\forall X \in \cat
\end{align*}
In particular, we have $F(V) \simeq V'$. For $\cat=\Rep(G)$ this implies moreover:
$$E_F({\ocat_g^\chi})=\ocat_{g'}^{\chi'}
\;\:\Rightarrow \;\;
F(\Ind_{\Cent(g)}^G(\chi))\cong \Ind_{\Cent(g')}^G(\chi')$$

Thus, possible images of $E_F$ are determined by the character table of
$G$ and
induction-restriction table with $\Cent(g),\Cent(g')$. We continue for the special case $g=1$ to determine the possible images
$E_F(\ocat_1^\chi)$ and hence $E_F(\LL_{1,1})$. Our formula above implies:
$$F(\chi)=\Ind_{\Cent(g')}^G(\chi')$$
In particular, $\Ind_{\Cent(g')}^G(\chi')$ has to be irreducible.


\subsection{Elementary abelian groups}\label{sec_Fp_AutBr}~\\

For $G=\ZZ_p^n$ with $p$ a prime number. We fix an isomorphism $\ZZ_p \simeq \widehat{\ZZ}_p$. We know that
    $$\BrPic(\Rep(\ZZ_p^n)) \simeq \mathrm{O}^+_{2n}(\F_p)$$
		where $\mathrm{O}^+_{2n}(\F_p):=\mathrm{O}_{2n}(\F_p,q)$ is the group of invertible $2n \times 2n$ matrices invariant under the form: 
		$$q:\F_p^n \times \F_p^n \to \F_p: (k_1,...,k_n,l_1,...,l_n) \mapsto \sum_{i=1}^n k_il_i $$
    For abelian groups, all $2$-cocycles over $\DG$ are lazy and the
    results of this article gives a product decomposition of
    $\BrPic(\Rep(\ZZ_p^n))$.
    \begin{itemize}
    \item $\cV \cong \GL_{n}(\F_p) \simeq \left\{ \begin{pmatrix} A^{-1} &0 \\ 0 &A \end{pmatrix} \mid A \in \GL_{n}(\F_p) \right\} \subset \mathrm{O}_{2n}(\F_p,q)$
    \item $\cB \cong B_{alt} \cong \left\{ \begin{pmatrix} \mathbbm{1}_n &B \\ 0 & \mathbbm{1}_n \end{pmatrix} \mid B = -B^T, B_{ii} = 0, B \in \F_p^{n \times n} \right\} \subset \mathrm{O}_{2n}(\F_p,q)$ 
\item $\cE \cong E_{alt} \cong \left\{ \begin{pmatrix} \mathbbm{1}_n &0 \\ E & \mathbbm{1}_n \end{pmatrix} \mid E = -E^T, E_{ii} = 0, E \in \F_p^{n \times n} \right\} \subset \mathrm{O}_{2n}(\F_p,q)$ 
\end{itemize}
    The set $\cRL/\sim$ consists of $n+1$ representatives $r_{[C]}$,
one for each possible dimension $d$ of a direct
    factor $\F_p^d\cong C\subset G$, and $r_{[C]}$ is an actual
reflection on the subspace $C$ with a suitable monoidal
    structure determined by the pairing $\lambda$. Especially the
generator $r_{[G]}$ conjugates $\cB$ and $\cE$. In this case the double coset decomposition is a variant of the
Bruhat decomposition of $\mathrm{O}_{2n}(\F_p,q)$. 
   
It is interesting to discuss how, in this example, our subgroups act on the Lagrangian subcategories and to see that this action is indeed transitive. $\mathbb{L}_0(G)=\mathbb{L}(G)$ is parametrized by pairs
$(N,[\mu])$ where $N$ is a subvector space of $\F_p^n$ and $[\mu] \in \H^2(N,k^\times)$ is
uniquely determined by an alternating bilinear form $\langle,\rangle_\mu$ 
on $N$ given by $\langle g,h \rangle_\mu = \mu(g,h)\mu(h,g)^{-1}$. Let $N'$ be the orthogonal complement, so $\F_p^n=N\oplus N'$. We have
$$\LL_{N,\mu}=\left\langle\ocat_g^{\chi_{N'}\langle g,-\rangle} \mid g\in N, \chi_{N'}\in\widehat{N}' \right\rangle \qquad \LL_{1,1}=\left\langle\ocat_1^{\chi} \mid \chi \in\widehat{G} \right\rangle  $$
$\bullet$ Elements in $\cVTildeL=\Out(G)=\GL_n(\F_p)$ stabilize $\LL_{1,1}$. \\
$\bullet$ For any $\delta$, a partial dualization $r_N \in \cRL$ on $N$ maps $\LL_{1,1}$ to
$\LL_{N,1}$. \\
$\bullet$ $b \in \Hom_{alt}(G,\widehat{G}) \simeq \cBTildeL$ acts by
$\ocat_{g}^\chi \mapsto \ocat_g^{\chi\cdot b(g,\cdot)}$.
In particular, it stabilizes $\LL_{1,1}$ and sends $\LL_{N,1}\mapsto
\LL_{N,\beta\mid_N}$ where $\beta \in \Z^2(G,k^\times)$ is uniquely (up to coboundary) determined by $b(g,h)=\beta(g,h)\beta(h,g)^{-1}$. \\
$\bullet$ $a \in \Hom_{alt}(\widehat{G},G) \simeq \cETildeL$ acts by  
$\ocat_{g}^\chi \mapsto \ocat_{a(\chi)g}^{\chi}$.
In particular, it sends $\LL_{1,1}$ to $\LL_{N,\eta}$ with $N=\im(a)$ being the image of $a$ and $\eta \in \Z^2(N,k^\times)$ is uniquely (up to coboundary) determined by $\eta(n,n')\eta(n',n)^{-1} = \chi(n')$ where $a(\chi) =n$ and $n' \in N=\im(a)$. For another $\chi'$ with $a(\chi')=n$ we have $\chi(n') = \chi(a(\rho)) = \rho(a(\chi))^{-1} = \rho(a(\chi')) = \chi'(n)$. \\

We see that can get every $\LL_{N,\mu} \in \mathbb{L}_0(G)$ with suitable combinations of elements of our subgroups applied to $\LL_{1,1}$. \\ 

\subsection{Simple groups}~\\

Let $G$ be a simple group, then our result returns
\begin{itemize}
\item $\cVTildeL=\Out(G)$
\item $\cBTildeL=\widehat{G}_{ab}\wedge\widehat{G}_{ab}=1$
\item $\cETildeL = Z(G) \wedge Z(G)=1$
\item $\cRL = 1$
\end{itemize}
hence the only \emph{lazy} autoequivalences are induced by outer
automorphisms of $G$. \\

We have no normal abelian subgroups except
$\{1\}$ and hence the only Lagrangian subcategory is $\LL_{1,1}$ and the
stabilizer $\Out(G)\ltimes \H^2(G,k^\times)$ is equal to
$\Aut_{br}(\DG\md\mod)$. \\

Observe that in this example we obtain also a decomposition of the full
Brauer-Picard group and our Question \ref{q_decomposition} is answered 
positively:
Namely, $\Aut_{br}(\DG\md\mod)$ is equal the image of the induction
$\Ind_{\Vect_G}$, while the other subgroups are trivial. \\

\subsection{Lie groups and quasisimple groups}~\\

Lie groups over finite fields $G(\F_{q}),q=p^k$ have (with small
exceptions) the
property $G_{ab}=1$ and there are no semidirect factors. On the other
hand, they may contain a nontrivial center
$Z(G)$. This is comparable to their complex counterpart, where the center of the
simply-connected
form $Z(G_{sc}(\CC))$ is equal to the fundamental group
$\pi_1(G_{ad}(\CC))$ of
the adjoint form with no center $Z(G_{ad}(\CC))=1$. In exceptional cases
for $q$, the maximal central extension may be larger than $\pi_1(G_{ad}(\CC))$. Similarly, we could consider 
central extensions of sporadic groups $G$; these appear in any insolvable group as part of the Fitting group.

\begin{definition}
  A group $G$ is called \emph{quasisimple} if it is a perfect central
extension of a simple group:
  $$Z\to G\to H\qquad Z=Z(G), \quad [G,G]=G$$
\end{definition}

As long as $\H^2(Z,\CC^\times)=1$, e.g. because $Z$ is cyclic, there is no
difference to the simple case. \emph{Nontrivial} $\cETildeL$-terms appear as
soon
as $\H^2(Z,\CC^\times)\neq 1$. This is the case for
$D_{2n}(\F_q)=\SO_{4n}(\F_q)$ (for $q$ odd or $q=2$) where we have $\pi_1(G_{ad}(\CC)) =\ZZ_2\times\ZZ_2$ and in some other
(exceptional) cases. We consider all universal perfect central extensions where $\H^2(Z,\CC^\times)\neq 1$: \\

\begin{center}
\begin{tabular}{lcl|l}
$Z$ &  & $H$ & $\cETildeL$ \\
\hline
$\ZZ_2\times \ZZ_2$ && $D_{2n}(\F_{q})$ & $\ZZ_2$ \\
\hline
$\ZZ_4\times \ZZ_4 \times \ZZ_3$ && $A_2(\F_{2^2})$ & $\ZZ_4$ \\
$\ZZ_3\times \ZZ_3 \times \ZZ_4$ && ${^2}A_3(\F_{3^2})$ & $\ZZ_3$ \\
$\ZZ_2\times \ZZ_2 \times \ZZ_3$ && ${^2}A_5(\F_{2^2})$  & $\ZZ_2$ \\
$\ZZ_2\times \ZZ_2$ && ${^2}B_2(\F_{2^3})$ & $\ZZ_2$ \\
$\ZZ_2\times \ZZ_2 \times \ZZ_3$ && ${^2}E_6(\F_{2^2})$ & $\ZZ_2$
\end{tabular}
\end{center}
The upper indices denote the order of the automorphism by which the so-called Steinberg groups are defined. $\Out(H)$ typically consists of scalar- and Galois-automorphisms of the base
field $\F_q$, extended by the group of Dynkin diagram automorphisms; for example $D_4$ we have the triality automorphisms $\SS_3$. Note further that any
automorphism on $G$ preserves the center $Z$, hence it factors to an
automorphism in $H$. The kernel of this group homomorphism
$\Out(G)\to\Out(H)$
is trivial, since all elements in $Z$ are products of commutators of $G$
elements. We have $\Out(G) \cong \Out(H)$ where surjectivity follows from $G$ being a universal central extension. For $G$ as above, the following holds:

\begin{itemize}
\item $\cVTildeL=\Out(H)$
\item $\cBTildeL=\widehat{G_{ab}}\wedge\widehat{G_{ab}}=1$
\item $\cETildeL=(\ZZ_n \times \ZZ_n \times \ZZ_k )\wedge (\ZZ_n \times \ZZ_n \times \ZZ_k )=\ZZ_n$ for gcd$(n,k)=1$ \\ where $n\in\{2,3,4\}$ as
indicated in the above table.
\item $\cRL=1$, as there are no direct factors of $G$.
\end{itemize}
Hence $\Aut_{br,L}(\DG\md\mod) \simeq \Out(H) \ltimes \ZZ_n$.  


\begin{claim}
The decomposition we proposed in Question \ref{q_decomposition} is also true
for the full Brauer-Picard group for the $G$ above. More precisely
\begin{align*}
\BrPic(\Rep(G))
&=\im(\Ind_{\Vect_G})\cdot \im(\Ind_{\Rep(G)}) \cdot \cR \\
&=\Out(G)\ltimes\H^2(G,k^\times)\cdot \ZZ_n\cdot 1
\end{align*}
\begin{itemize}
\item $\im(\Ind_{\Vect_G})=\H^2(G,k^\times)$
\item $\im(\Ind_{\Rep(G)}) \simeq \widetilde{\cE} \simeq \cETildeL \simeq \ZZ_n$
\item No reflections, as there is no semidirect decomposition of $G$.
\end{itemize}
\end{claim}
\begin{proof}
Let $N$ be a normal abelian subgroup of $G$. Then the image $\pi(N)$ of the surjection $\pi:G \to H$ is a normal abelian subgroup of $H$. Since $H$ is simple and non-abelian, $N$ has to be a subgroup of the center $\ker(\pi)=Z=Z(G)$. Further, since $\widehat{G} = \widehat{G_{ab}}=1$, the only $1$-dimensional simple object in $\LL_{1,1}$ is $\ocat^1_1$. On the other hand, if $[\mu] \in \H^2(N,k^\times)$ is degenerate on a non-trivial $N$, hence if there exists a $n \in N$ such that $\mu(n,\cdot)\mu(\cdot)^{-1}=1$, then $\LL_{N,\mu}$ has at least two (non-isomorphic) $1$-dimensional simple objects, namely $\ocat_1^1$ and $\ocat^1_n$. This implies, all $\LL_{N,\mu} \in \mathbb{L}_0(G)$ with $N \neq 1$ must have a non-degenerate $\mu$. \\
Recall that an element in $\im(\Ind_{\Rep(G)})$ determined by an $a \in \Hom_{alt}(\widehat{Z},Z)$ sends $\ocat^\chi_1$ to $\ocat^\chi_{a(\chi')}$ where $\chi':Z \to k^\times$ the $1$-dimensional character determined by $\chi$ restricted to $Z$. Given a pair $(N,[\mu])$ where $N$ is a normal central subgroup of $G$ and $\mu$ non-degenerate we give $a \in \Hom_{alt}(\widehat{Z},Z)$ such that $a(\LL_{1,1})=\LL_{N,\mu}$. Since $\mu$ is non-degenerate, $b:N \ito \widehat{N}$ defined by $b(n)(n')=\mu(n,n')\mu^{-1}(n',n)$ is bijective. We claim that $a \in \Hom_{alt}(\widehat{Z},Z)$ defined by $a(\chi) := b^{-1}(\chi|_N)$ does the job. Using that $N$ is a central normal subgroup, we see that $N = \im(a)$. Also, $\mu$ is indeed the cocycle determined by $a$ because: $\chi(n') = b(n,n') = \mu(n,n')\mu^{-1}(n',n)$ for $a(\chi)=n$ and all $n' \in N$. This proves that $a(\LL_{1,1})=\LL_{N,\mu}$. The action of $\Hom_{alt}(\widehat{Z},Z) \simeq \ZZ_n$ on $\mathbb{L}_0(G)$ is therefore indeed transitive. All elements of $\ZZ_n$ act differently on $\mathbb{L}_0(G)$ and therefore $\widetilde{\cE}_L \simeq \cE_L$. Also, in this case, the lazy elements $\cE_L \simeq \ZZ_n$ already give all $\im(\Ind_{\Rep(G)}) \simeq \ZZ_n$. The only non-lazy terms come from $\im(\Ind_{\Vect_G})=\H^2(G,k^\times)$ in the stabilizer.      

%
\end{proof}

\subsection{Symmetric group $\SS_3$}~\\

For $G=\SS_3$ the following holds
\begin{itemize}
\item $\cVTildeL=\Out(\SS_3)=1$
\item $\cBTildeL=\hat{\SS}_3\wedge\hat{\SS}_3=\ZZ_2\wedge\ZZ_2=1$
\item $\cETildeL=Z(\SS_3)\wedge Z(\SS_3)=1$
\item $\cRL=1$, as there are no direct factors.
\end{itemize}
Hence our result implies that there are no \emph{lazy} braided autoequivalences of $D\SS_3\md\mod$.

\noindent
We now discuss the full Brauer-Picard group of $\SS_3$ which
was computed in \cite{NR14} Sec. 8.1: We have the Lagrangian subcategories
$\LL_{1,1},\LL_{\langle(123)\rangle,1}$ and stabilizer $\Out(\SS_3)\ltimes
\H^2(\SS_3,k^\times)=1$. Hence $\Aut_{br}(D\SS_3\md\mod)=\ZZ_2$.

\begin{claim}
The decomposition we proposed in Question \ref{q_decomposition} is also true
for the full Brauer-Picard group of $\SS_3$. More precisely
\begin{align*}
\BrPic(\Rep(\SS_3))
&=\im(\Ind_{\Vect_G})\cdot \im(\Ind_{\Rep(G)})\cdot \cR = 1\cdot1\cdot \ZZ_2
\end{align*}
\begin{itemize}
\item $\im(\Ind_{\Vect_G})=1$
\item $\im(\Ind_{\Rep(G)})=1$
\item Reflections $\ZZ_2$, generated by the partial dualizations $r_N$
on the
semidirect decomposition $\SS_3=\ZZ_3\rtimes\ZZ_2$ with abelian normal
subgroup
$N=\ZZ_3$. More precisely $r$ interchanges
$\LL_{1,1},\LL_{\langle(123)\rangle,1}$, the action on $\ocat_1^\chi$ is made
explicit in the proof.
\end{itemize}
\end{claim}

\begin{proof}
First, $\im(\Ind_{\Vect_G})$ is the stabilizer $\Out(\SS_3)\ltimes
\H^2(\SS_3,k^\times)=1$. Second, \cite{Dav01} states that
$\Aut_{mon}(\Rep(G))$ is a subset of the set of pairs consisting of a
abelian normal subgroup and a \emph{non-degenerate} $G$-invariant
cohomology class on this subgroup. The only nontrivial normal abelian
subgroup for $\SS_3$ is cyclic and hence there is no such pair, thus  
$\im(\Ind_{\Rep(\SS_3)})=1$.

We apply the general considerations in Section \ref{sec_nonlazyReflection}: The Clifford
decomposition of the restrictions triv$|_N$, sgn$|_N$, ref$|_N$ to $N=\ZZ_3$ is
$1,1,\zeta\oplus\zeta^2$ respectively. In the last case $\ZZ_2$ is acting by
interchanging the summands (resp. by Galois action), the inertia group being
trivial. We get $r(\ocat_1^{\mref})=\ocat_{(123)}^1$ and the partial dualization $r$ maps
$$\LL_{1,1}=\left\langle\ocat_1^{\mathrm{triv}},\;\ocat_1^{\mathrm{sgn}},\;\ocat_1^{\mathrm{ref}}\right\rangle
\longmapsto
\LL_{\langle(123)\rangle,1}=\left\langle
\ocat_1^{\mathrm{triv}},\;\ocat_1^{\mathrm{sgn}},\;\ocat_{(123)}^{1}\right\rangle$$
\end{proof}


\subsection{Symmetric group $\SS_4$}~\\

For $G=\SS_4$ the following holds:
\begin{itemize}
\item $\cVTildeL=\Out(\SS_4)=1$
\item $\cBTildeL=\hat{\SS}_4\wedge\hat{\SS}_4=\ZZ_2\wedge\ZZ_2=1$
\item $\cETildeL=Z(\SS_4)\wedge Z(\SS_4)=1$
\item $\cRL=1$, as there are no direct factors.
\end{itemize}

Hence, your result implies that there are no \emph{lazy} braided autoequivalences of $D\SS_4\md\mod$.\\

The full Brauer-Picard group of $\SS_4$ was computed in Sec. 8.2. \cite{NR14}. Denote the standard irreducible representations of $\SS_4$ by
triv, sgn, ref2, ref3, ref3$\otimes$sgn, where $\mref2$ and $\mref3$ are the standard two and three dimensional irreducible representations of $\SS_4$. There is a unique abelian normal
subgroup $N=\{1,(12)(34),(13)(24),(14)(23)\}\cong \ZZ_2\times \ZZ_2$. We have three Lagrangian subcategories
$\LL_{1,1},\LL_{N,1},\LL_{N,\mu}$ for
$N \cong \ZZ_2\times \ZZ_2$. The stabilizer is $\Out(\SS_4)\ltimes
\H^2(\SS_4,k^\times)=\ZZ_2$. In particular, $\Aut_{br}(D\SS_4\md\mod)$
has order $6$. One checks, that the nontrivial $[\beta] \in \H^2(\SS_4,k^\times)$
restricts to the nontrivial $[\mu]$ on $N$, hence
$$[\beta]:\; \LL_{1,1},\LL_{N,1},\LL_{N,\mu}
\longmapsto \LL_{1,1},\LL_{N,\mu},\LL_{N,1}$$
and by order and injectivity, we have $\Aut_{br}(D\SS_4\md\mod)\cong\SS_3$.

\begin{claim}
The decomposition we proposed in Question \ref{q_decomposition} is also true
for the full Brauer-Picard group of $\SS_4$. More precisely
\begin{align*}
\BrPic(\Rep(\SS_4))
&=\im(\Ind_{\Vect_G})\cdot \im(\Ind_{\Rep(G)})\cdot \cR\\
&=\ZZ_2\cdot \ZZ_2\cdot \ZZ_2 = \SS_3
\end{align*}
\begin{itemize}
\item $\im(\Ind_{\Vect_G})=\ZZ_2$ generated by the nontrivial cohomology
class
$[\beta]$ of $\SS_4$ with action on $\mathbb{L}_0(G)$ described above. Note
that
$[\beta]$ restricts to the unique nontrivial cohomology class $[\mu]$ on $N$.
\item $\im(\Ind_{\Rep(G)})=\ZZ_2$ generated by the non-lazy monoidal
autoequivalence $F$ of $\Rep(\SS_4)$, described in detail in in Sect. 8 of \cite{Dav01}. 
$E_F\in \Aut_{br}(\DG\md\mod)$ interchanges $\LL_{1,1},\LL_{\langle(123)\rangle,\mu}$.
\item Reflections $\cR\cong \ZZ_2$, generated by the reflection $r=r_N$ on the
semidirect decomposition $\SS_4=N\rtimes\SS_3$ with abelian kernel $N$.
More precisely $r$ interchanges $\LL_{1,1},\LL_{\langle(123)\rangle,1}$.
\end{itemize}

\end{claim}
\begin{proof}
The stabilizer $\im(\Ind_{\Vect_G})$ and its action on $\mathbb{L}_0(G)$ has
already been calculated. To compute $\im(\Ind_{\Rep(\SS_4)})$ note that $\Aut_{mon}(\Rep(\SS_4))$ has
been explicitly computed in Sect. 8 of \cite{Dav01}: Since
there is only one ontrivial normal subgroup $N=\ZZ_2\times \ZZ_2$ and only one (up to coboundary)
non-degenerate $2$-cocycle $\mu$ on $N$, which is $G$-invariant
\emph{only}
as a cohomology class $[\mu]$. In \cite{Dav01} it is shown that this gives rise to a
(non-lazy) monoidal autoequivalence $F$ of $\Rep(\SS_4)$ such that $F(\mref3)=\mref3\otimes \sgn$ which corresponds to mapping $[(12)]$ to $[(1234)]$. 
This automorphism is visible as a symmetry of the character table. \\
We compute the action of $E_F\in \im(\Ind_{\Rep(\SS_4)})$ on all $\ocat_1^\chi$. First, $\chi=\triv,\sgn,\mref2$ restricted to $N$ are trivial
representations. Second, the possible images
$$E_F(\ocat_1^{\mref3})=\ocat_g^\chi,\quad
E_F(\ocat_1^{\mref3\otimes \sgn})=\ocat_{g'}^{\chi'}$$
belong to the $G$-conjugacy classes in $N$, i.e. $g,g'
=1$ or $g,g'
=(12)(34)$. They have to fulfill the characterization outlined in
general considerations above, namely:
$$F(\mref)=\mref\otimes \sgn\stackrel{!}{=}\Ind_{\Cent(g)}^G(\chi) \quad\quad F(\mref\otimes \sgn)=\mref\stackrel{!}{=}\Ind_{\Cent(g')}^G(\chi')$$
Assume $g=g'=1$. This implies that $E_F(\mathbb{L}_0(G))=\mathbb{L}_0(G)$ and thus $E_F$ is in the stabilizer,
which is $\Out(\SS_4)\ltimes \H^2(\SS_4,k^\times)$. This is not possible
since $E_F$ acts nontrivial on objects and does not come from an automorphism of $G$.
Therefore we take $g,g'
=(12)(34)$ and consider
$$F(\mref3)=\mref3\otimes \sgn\stackrel{!}{=}\Ind_{\Cent(12)(34)}^G(\chi) \quad\quad F(\mref3 \otimes \sgn)=\mref 3\stackrel{!}{=}\Ind_{\Cent(12)(34)}^G(\chi')$$
where $\Cent(12)(34)=\langle (12),(13)(24)\rangle\cong\DD_4$. The
character
table quickly returns the only possible $\chi,\chi'$: 
$$E_F(\ocat^{\mref3}_1)=\ocat_{(12)(34)}^{(--)}\qquad
E_F(\ocat^{\mref3 \otimes \sgn}_1)=\ocat_{(12)(34)}^{(+-)}\qquad$$
where $(++),(--),(+-),(--)$ are the four $1$-dimensional irreducible representations of $\DD_4 = \langle (12),(13)(24)\rangle$ where the first generator acts by the first $\pm 1$ in the bracket and the second generator by the second $\pm 1$. We see that $\chi|_N$ and $\chi'|_N$ are nontrivial, hence in
$\LL_{N,\mu}$ for
$\mu$ nontrivial and
$$E_F:\; \LL_{1,1}=\left\langle\ocat_1^{\triv},\ocat_1^{\sgn},\ocat_1^{\mref2},
\ocat_1^{\mref3},\ocat_1^{\mref3\otimes \sgn}\right\rangle$$
$$\longmapsto
\LL_{N,\mu}=\left\langle\ocat_1^{\triv},\ocat_1^{\sgn},\ocat_1^{\mref2},
\ocat_{(12)(34)}^{(--)},\ocat_{(12)(34)}^{(+-)}\right\rangle$$

We finally calculate the action of the partial dualization $r$ on the
decomposition $\SS_4=N \rtimes \SS_3$. The general considerations in Section
\ref{sec_nonlazyReflection} imply the following for the images
$r(\ocat_1^\chi)$: Since $\chi=\triv,\sgn,\mref2$ restricted to $N$ are trivial, these are fixed. For $\chi=\mref3,\chi'=\mref3\otimes
\sgn$ the restrictions are easily determined by the character table to be
$$\chi|_N=\chi'|_N=(-+)\oplus (+-)\oplus (--)$$
which returns via $\delta:kN\to k^N$ precisely the conjugacy class
$[(12)(34)]$ and the inertia subgroup is $I=N\rtimes
\langle(12)\rangle$. To see
the action on the centralizer, we restrict the representations
$\chi,\chi'$ to
$I$ and extend it trivially to $I=\Cent(12)(34)=\langle
(12),(13)(24)\rangle\cong\DD_4$ yielding finally:
$$r(\ocat_1^{ref})=\ocat_{(12)(34)}^{(++)}
\qquad r(\ocat_1^{\mref\otimes \sgn})=\ocat_{(12)(34)}^{(-+)}$$ 
$$r:\; \LL_{1,1}=\left\langle\ocat_1^{\triv},\ocat_1^{\sgn},\ocat_1^{\mref2},
\ocat_1^{\mref3},\ocat_1^{\mref3\otimes \sgn},\right\rangle$$
$$\longmapsto
\LL_{N,1}=\left\langle\ocat_1^{\triv},\ocat_1^{\sgn},\ocat_1^{\mref2},
\ocat_{(12)(34)}^{(++)},\ocat_{(12)(34)}^{(-+)},\right\rangle$$

%

\end{proof}

\noindent{\sc Acknowledgments}: 
We are grateful to C. Schweigert for many helpful
discussions. The authors are partially supported by the DFG Priority 
Program SPP 1388 ``Representation Theory'' and the Research Training Group 1670 ``Mathematics Inspired by String Theory and QFT''. 
S.L. is currently on a research stay supported by DAAD PRIME, funded by BMBF and EU Marie Curie Action.

\end{document}